\documentclass[11pt]{article}

\usepackage{ifpdf}   

\ifpdf  
\usepackage{graphicx,pstricks,xypic,comment}
\else    
\usepackage{inputenc}
\usepackage[hypertex]{hyperref}      

\usepackage[active]{srcltx}          
\usepackage{graphicx,pstricks,xypic,comment}
\DeclareGraphicsExtensions{.eps,.ps,.pst}        

\fi     

\newcommand{\executeiffilenewer}[3]{%
 \ifnum\pdfstrcmp{\pdffilemoddate{#1}}%
 {\pdffilemoddate{#2}}>0%
 {\immediate\write18{#3}}\fi%
}

\usepackage{color}
\usepackage{mathrsfs} 
\usepackage{amssymb,amscd,amsmath,amsfonts,amsthm}

\theoremstyle{plain}

\newtheorem{theo}{Theorem}[section]
\newtheorem{lem}[theo]{Lemma}
\newtheorem*{lem*}{Lemma}

\newtheorem{prop}[theo]{Proposition}
\newtheorem{conj}[theo]{Conjecture}
\newtheorem*{conj*}{Conjecture}

\theoremstyle{remark}

\theoremstyle{definition}

\newcommand{\ga}{\gamma}
\newcommand{\Si}{\Sigma}

\newcommand{\al}{\alpha}
\newcommand{\De}{\Delta}
\newcommand{\si}{\sigma}
\newcommand{\om}{\omega}
\newcommand{\Ga}{\Gamma}
\newcommand{\be}{\beta}
\newcommand{\la}{\lambda} 
\newcommand{\Om}{\Omega}

\newcommand{\ep}{\varepsilon} 
\newcommand{\Hilb}{{\mathcal{H}}}

\newcommand{\C}{\mathbb C}
\newcommand{\R}{\mathbb R}
\newcommand{\N}{\mathbb N}
\newcommand{\Z}{\mathbb Z}
\newcommand{\T}{\mathbb T}

\newcommand{\su}{\operatorname{SU}_2}

\newcommand{\sldeux}{\operatorname{Sl}_2(\mathbb{Z})}
\newcommand{\mo}{{\mathcal{M}}}  
\newcommand{\tr}{\operatorname{Tr}}
\newcommand{\id}{\operatorname{id}}

\newcommand{\Ci}{{\mathcal{C}}^{\infty}}  
\newcommand{\End}{\operatorname{End}} 

\newcommand{\alt}{\operatorname{alt}} 

\newcommand{\ab}{\operatorname{ab}} 
\newcommand{\ir}{\operatorname{irr}}

\newcommand{\MS}{\operatorname{MS}}

\newcommand{\sL}{{\mathcal{L}}} 
\newcommand{\sM}{{\mathcal{M}}} 
\newcommand{\sP}{{\mathcal{P}}} 
\newcommand{\sT}{{\mathcal{T}}} 
\newcommand{\sA}{\mathcal{A}}

\newcommand{\bpi}{\pi_{\mathbb{C}}}

\title{Knot state asymptotics I\\ AJ Conjecture and abelian representations}

\author{L. Charles\footnote{Institut de
    Math{\'e}matiques de Jussieu (UMR 7586), Universit{\'e} Pierre et
    Marie Curie -- Paris 6, Paris, F-75005 France.}\,\, and J. March{\'e}\footnote{Centre de math{\'e}matiques Laurent Schwartz (UMR 7640), Ecole Polytechnique -- 91128 Palaiseau, France} }

\date{}
\begin{document}

\maketitle
\begin{abstract}
Consider the Chern-Simons topological quantum field theory with gauge group $\su$ and level $k$. Given a knot in the 3-sphere, this theory associates to the knot exterior an element in a vector space. We call this vector the knot state and study its asymptotic properties when the level is large.  

The latter vector space being  isomorphic to the geometric quantization of the $\su$-character variety of the peripheral torus, the knot state may be viewed as a section defined over this character variety. We first conjecture that the knot state concentrates in the large level limit to the character variety of the knot. This statement may be viewed as a real and smooth version of the AJ conjecture. 
Our second conjecture says that the knot state in the neighborhood of abelian representations is a Lagrangian state. 

Using microlocal techniques, we prove these conjectures for the figure  eight and torus knots. The proof is based on $q$-difference relations for the colored Jones polynomial. 
We also provide a new proof for the asymptotics of the Witten-Reshetikhin-Turaev invariant of the lens spaces and a derivation of the Melvin-Morton-Rozansky theorem from the two conjectures.
\end{abstract}

\section{Introduction}

Chern-Simons theory is a topological field theory dealing with surfaces and three dimensional cobordisms. In the classical part of the theory, we associate to each surface $\Si$ a symplectic manifold $\mo_G (\Si)$ and to each three dimensional manifold $M$ with boundary $\Si$ a Lagrangian immersion $\mo_G (M) \rightarrow  \mo_G ( \Si )$. Here $G$ is a compact Lie group and for any compact manifold $P$, $\mo_G ( P)$ is the moduli space of representations of the fundamental group of $P$ in $G$ up to conjugation. 
On the quantum side, we associate to $\Si$ and $M$ a vector space $V_{G,k} ( \Si)$ and a state $Z_{G,k} (M) \in V_{G,k}( \Si)$. The parameter  $k$ is an arbitrary positive integer called the level.

Following Witten (\cite{witten}), the space $V_{G,k} ( \Si)$ may be defined as the geometric quantization of $\mo_G ( \Si)$ and the state $Z_{G,k} (M)$ as a partition function with Chern-Simons action. These definitions are heuristic because they are based on functional integration. Using topological methods, Reshetikhin and Turaev gave a rigorous mathematical construction in \cite{rt}. The spaces and states obtained  satisfy some functorial properties as do the moduli spaces $\mo_G (M)$ and  $\mo_G ( \Si)$, but apart from that, they seem to be completely unrelated to the classical data. 
Later it was established that the vector space geometrically quantizing $\mo_G ( \Si)$ and the  $V_{G,k} ( \Si)$ of \cite{rt} have the same dimension, given by the Verlinde formula, cf. the review article \cite{sorger}. Much more can be said in the case $\Si$ is a torus, cf. for instance \cite{je,an,gu}. 
The relation between the states $Z_{G,k}( M)$  and the moduli spaces $\mo_G ( M)$ remains mysterious. 

The correspondence between classical and quantum mechanics relates Lagrangian submanifolds of the classical phase space to states of the quantum Hilbert space. In particular, in semiclassical analysis, the meaning of this correspondence is that the asymptotic behavior of some quantum states is encoded by a Lagrangian submanifold.

It is natural to conjecture that the states $Z_{G,k} (M)$ correspond to the Lagrangians  $\mo_G ( M)$ in the sense of semi-classical analysis. We will develop in this paper and its companion \cite{LJ2} the case of knot exteriors. Furthermore, we will show that the AJ conjecture, the Melvin-Morton-Rozansky theorem and the Witten asymptotic conjecture take place naturally in that picture. On this basis, we also offer some useful tools to attack this last conjecture. At the end of the second article, we will prove Witten asymptotics conjecture for most Dehn fillings of the figure eight knot. Note that the conjecture has only been proved for Seifert manifolds and some mapping tori, cf. \cite{LJ2} for discussion and references.

Let us give some details. From now on we only consider the group $G = \su$.  Let $K$ be a knot in $S^3$. Let $E_K$ be the complement of an open tubular neighborhood of $K$ so that the boundary of $E_K$ is a 2-dimensional torus $\Sigma$. We define the knot state as the family $$(Z_k(E_K)\in V_k(\Sigma), k \in \Z_{>0} ).$$
Here the $Z_k$'s and $V_k$'s are the ones constructed in \cite{bhmv} where we disregard the anomaly correction.  
The geometric quantization of $\mo (\Si)$ can be carried out as follows. Let  $E$ be the vector space $H_1(\Sigma,\R)$ and $R$ be its lattice $H_1(\Sigma,\Z)$. Then $\mo (\Si)$ is isomorphic to the quotient of $E$ by $R\rtimes \Z_2$ where $\Z_2$ acts by $\pm \id_E$. The canonical projection is the map 
$$ \pi:E\to \mo(\Sigma) , \qquad  \pi(x)(\gamma)=\exp((\gamma\cdot x)D)$$ for any $x\in E,\gamma\in R$ where $\cdot$ stands for the intersection product and $D$ is the diagonal matrix with entries $2i\pi,-2i\pi$. 

 The quantization of $E/R$ at level $k$ is isomorphic to the space $\Hilb_k $  of $j$-holomorphic $R$-invariant sections of $L^k\otimes \delta$ over $E$. Here $L$ is a prequantum bundle over $E$, $j$ a linear complex structure and $\delta$ a half-form bundle, cf. Section \ref{sec:torus-quantization} for precise definitions. The space $\Hilb_k$ actually consists of theta series. The quantization of $\mo (\Si)$ is isomorphic to the subspace $\Hilb_k^{\alt}$ of alternating sections. Equivalently $\Hilb_k^{\alt}$ is the space of holomorphic sections of a genuine (orbi)-bundle over $E/(R \rtimes \Z_2)$.
It is known that there exists an isomorphism 
$$I_k: V_k(\Sigma)\sim \Hilb_k^{\alt}$$ natural with respect to the projective action of the mapping class group of $\Si$, here we use the construction provided in \cite{l2}.  We will prove that this isomorphism is uniquely defined up to a phase of the form $\pi (n/4 + n'/2k)$ with $n, n' \in \Z$, cf. Theorem \ref{sec:equiv-top-geom}. This allows us to consider the knot state as a family $(Z_k (E_K) \in \Hilb^{\alt}_k , k \in \Z_{>0})$ of holomorphic sections well-defined up to these phases.  

The first semi-classical invariant of a quantum state is its microsupport. Roughly speaking this is the subset of the phase space where the state lives. This notion was initially introduced by H{\"o}rmander for solutions of partial differential equation under the name of wave front set, cf. the discussion in \cite{hormander} p. 323.  The analog notion for the quantization of K{\"a}hler manifolds, which is relevant here, has been defined in \cite{l5}. 
We conjecture that the microsupport of the knot state is included in the image of the restriction map $r: \mo (E_K) \rightarrow \mo ( \Si)$. For the introduction, we state a weaker version of the actual Conjecture \ref{sec:conj_microsupport}.

\begin{conj}\label{conj:microsupport_intro}
Let $K$ be a knot in $S^3$.
For any $x \in  E \setminus \pi ^{-1} (r ( \mo (E_K))) $,  there is a a sequence of positive numbers $(C_M)$ such that for any $M$ and $k$,
$$|Z_k(E_K)(x)|\le C_M k^{-M} .$$
\end{conj}

Numerical evidence for the conjecture is provided by the representation of the norm of the knot state in Figure \ref{fig:husimi}.

\ifpdf
\begin{figure}\label{fig:husimi}
\begin{center}
\includegraphics[width=12cm]{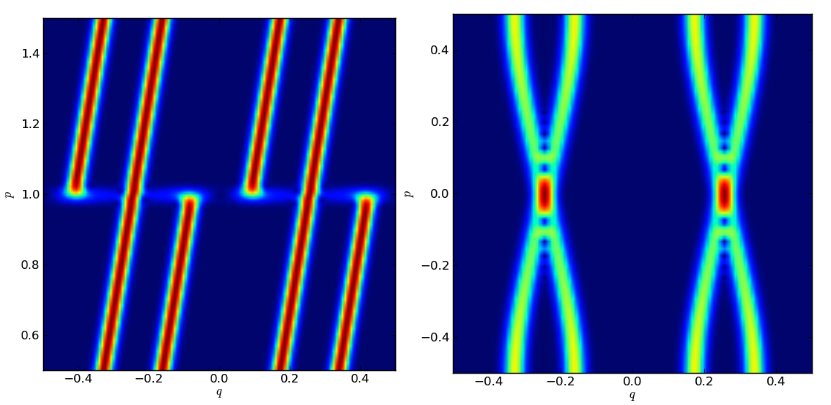}
\caption{The pointwise norm of the knot state for the trefoil and the figure eight knot at level $k=200$}
\end{center}
\end{figure}
\fi

In Theorem \ref{sec:8_theo}, we prove this conjecture for the figure eight knot, and in Theorem \ref{prop:torus_ab_preuve} we prove a weaker statement for the torus knots. The complete proof for torus knots will be given in \cite{Ctoric}. A stronger version of the conjecture is that the microsupport of the knot state is exactly $r ( \mo ( E_K))$. This version holds true for the figure eight knot but remains conjectural for torus knots.

Using the characterization of the microsupport in terms of Toeplitz operators (Proposition \ref{prop:micro_supp_Toeplitz}), we can view this conjecture as a smooth and real version of the AJ conjecture of Garoufalidis \cite{gar}. This will be explained with details in Section \ref{sec:relation-with-aj}. 

To complete this, we consider the asymptotic behavior of the knot state on $r( \mo ( E_K))$. We conjecture that the knot state is Lagrangian. Such states were introduced initially to construct solutions of linear differential equation and called WKB solutions. These functions have an oscillatory integral representation with a phase parameterizing appropriately a Lagrangian submanifold of the phase space \cite{duistermaat}.  
Here we use the definition proposed in \cite{l3} for the quantization of K{\"a}hler manifolds.  

The moduli space $\mo ( E_K)$ is the union of the subsets $\mo^{\ab}$ and $\mo^{\ir}$ consisting respectively of abelian and irreducible representations. In this paper we only consider the asymptotic behavior of the knot state on $r( \mo^{\ab})$, the irreducible part $r( \mo^{\ir})$ will be treated in \cite{LJ2}. Let $\lambda$ be the longitude of $K$, viewed as an element of $H_1(\Sigma,\Z)$. Since $H_1 ( E_K, \Z)$ is generated by a meridian, $ r ( \mo ^{\ab} (E_K))) = \pi (\la \R)$. 

\begin{conj} \label{conj:abelien_intro}
Let $K$ be a knot in $S^3$ and $\Delta_K$ be its Alexander polynomial. Let $q \in \R$ such that $\De_K ( e ^{4i \pi q} ) \neq 0$ and such that a neighborhood of $\pi ( q \la)$ does not meet $r( \mo ^{\ir})$. 
Then 
$$ Z_k (E_K) (q\la ) \sim \sqrt{2} e^{im\frac\pi 4}t_{\lambda}^k (q \la) \frac{\sin(2\pi q) }{\Delta_K (e^{4i\pi q})}\Omega_{\lambda} $$ 
where $m$ is an integer, $t_\la$ is a flat section of $L \rightarrow \R \la$ and $\Om_{\la}\in \delta$ is such that $\Om_\la ^2 ( \la) = 1$.   
\end{conj}

As for Conjecture \ref{conj:microsupport_intro}, this statement is a simplified version. The actual conjecture  gives the asymptotic behavior of the knot state uniformly on the neighborhood of $q \la$ for any $q$ satisfying the above assumption, cf. Conjecture \ref{conj:abelian}. So it describes the transition between the oscillatory sequence of Conjecture \ref{conj:abelien_intro} and the $O(k^{-\infty})$ sequence of Conjecture \ref{conj:microsupport_intro}. 

 In Theorems \ref{prop:torus_ab_preuve} and \ref{sec:8_theo}, we prove Conjecture 5.2 for the torus knots and the figure eight knot.

The scalar product of two Lagrangian states whose microsupports intersect transversally can be estimated with a pairing formula. From this we obtain a new proof of the asymptotic formula of Jeffrey \cite{je} for the WRT-invariants of the lens spaces. Furthermore we derive an analytic version of the Melvin-Morton-Rozansky conjecture from Conjecture 5.1 and 5.2. 
More precisely, we show that there exists $\delta >0$ such that for any $k$ and $\ell$ satisfying $ 1 \leqslant \ell \leqslant k  \delta$,    
 $$\tilde {J}_{\ell}^K(e^{-2i\pi/k})=\frac{1}{\Delta_K (e^{2i\pi\ell/k})}  + O( \ell^{-1})   $$
with a $O$ uniform with respect to $k$ and $\ell$. 
In this formula $\tilde J_\ell^K$ is the sequence of colored Jones polynomials normalized so that it equals $1$ for the unknot (see Section \ref{sec:jones}).
Our proof has the advantage over the previously known \cite{glmmr} in that it gives the optimal bound for $\delta$, it is any real number less than 
$$ \inf \bigl\{ q; \;  \Delta_K (e^{2i\pi q }) = 0  \text{ or } \exists \rho\in \mo^{\ir}(E_K)\text{ such that }\tr(\rho(\mu))=2\cos(2\pi q) \bigr\}$$
Here $\mu$ is a meridian of $K$. For instance, for the figure eight knot and the trefoil, this infimum is $\frac{1}{6}$. As we will see in \cite{LJ2}, the asymptotic behavior changes drastically when we overpass this bound.

In Section 2, we recall the basic ingredients of TQFT needed in this article, discuss the geometric quantization of the torus and relate it to TQFT. Section 3 is devoted to the semi-classical properties of curve operators and basis vectors. We introduce the knot state in Section 4 and provide formulas for the colored Jones functions of the figure eight and torus knot together with old and new q-difference relations. In Section 5, we state our main conjectures and prove them in the case of figure eight and torus knots. In the last section, we state the pairing formula and apply it to the WRT invariants of the lens spaces and to Melvin-Morton-Rozansky theorem. 
The paper ends with an appendix on Toeplitz operators.

{\bf Acknowledgements:}

We would like to thank Fr{\'e}d{\'e}ric Faure for his help in obtaining the figures presented in the introduction and his interest in this work. We also thank Gregor Masbaum and the ANR team "Quantum Geometry and Topology" for their interest and valuable discussions. The second author was supported by the French ANR project ANR-08-JCJC-0114-01.

\section{Geometrical versus topological constructions}\label{geomtopol}

\subsection{Topological quantum field theories (TQFT)}\label{tqft}

In this paper we work with the topological quantum field theory for the group $\su$ constructed in \cite{bhmv} from the Kauffman bracket. This theory provides a family $(V_k, Z_k), \; k \in \Z_{>0}$ of functors from cobordism categories to the category of complex vector spaces. The integer $k$ is called the level. The cobordism category at level $k$ is defined as follows: 
\begin{itemize}
\item[-] objects are pairs $(\Sigma,\nu)$ where $\Sigma$ is a closed oriented surface and $\nu$ is a linear Lagrangian subspace of $H_1(\Sigma,\R)$.
\item[-] a morphism between two objects $(\Sigma_0,\nu_0)$ and $(\Sigma_1,\nu_1)$ is a 4-tuple $(M,L,c,n)$ where $M$ is an oriented 3-manifold whose boundary is identified with $-\Sigma_0\amalg\Sigma_1$, $L$ is a banded link in $M$ whose components are assigned a color $c:\pi_0(L)\to\{1,\ldots,k-1\}$ and $n$ is an integer. 
\end{itemize}

The composition of two morphisms $(M,L,c,n):(\Sigma_0,\nu_0)\to(\Sigma_1,\nu_1)$ and $(M',L',c',n'):(\Sigma_1,\nu_1)\to(\Sigma_2,\nu_2)$ is given by $(M\cup_{\Sigma_1} M',L\amalg L',c\amalg c',n'')$ where 
\begin{gather} \label{eq:comp_maslov}
n''=n+n'-\mu(\nu_M \circ \nu_0,\nu_1,\nu_2 \circ \nu_{M'}).
\end{gather}
In this formula,  $\nu_M$ is the kernel of the map $H_{1} (\Sigma_0) \oplus H_1( \Sigma_1) \rightarrow H_1(M)$ induced by the inclusion, $\nu_{M'}$ is defined similarly and    
$\mu$ is the Maslov index of a triple of Lagrangians in $H_1(\Sigma_1)$. 
Here instead of using the notion of $p_1$-structure, we followed the approach of Walker to resolve anomalies. We refer to \cite{gm} for a detailed discussion.

The TQFT associates to any object $(\Sigma,\nu)$ a Hermitian vector space $V_k(\Sigma,\nu)$ and to any morphism $(M,L,c,n)$ between $(\Sigma_0,\nu_0)$ and $(\Sigma_1,\nu_1)$ a morphism 
$$Z_k(M,L,c,n)\in\textrm{Hom}(V_k(\Sigma_0,\nu_0),V_k(\Sigma_1,\nu_1)).$$
 This correspondence is functorial, it maps disjoint unions to tensor products and orientation reversal to complex conjugation. 

By construction one has for any morphism $(M,L,c,n)$ and integer $n$ the equality 
\begin{gather} \label{eq:def_tauk} 
Z_k(M,L,c,n)=\tau_k^nZ_r(M,L,c,0) \qquad \text{with } \qquad \tau_k= e^{\frac{3i\pi}{4}- \frac{3i\pi}{2k}}.
\end{gather}
As we are interested in the asymptotics $k\to\infty$, we see that ignoring the anomaly will cause at first order an indeterminacy in  the group of 8-th roots of unity.

In the sequel, we mainly work with a torus $\Sigma$. Choose a Lagrangian $\nu \in H^1( \Si)$ and let $V_k ( \Si ) = V_k ( \Si , \nu)$.  Fix an oriented diffeomorphism $\phi:\Sigma\to S^1\times S^1$. 
Then the manifold $D^2\times S^1$, the banded link $x=[0,1/2]\times S^1$ and the color $\ell\in\{1,\ldots,k-1\}$ gives a vector 
$$e_\ell :=Z_k(D^2 \times S^1, x,\ell,0) \in V_k ( \Si) .$$ 
For any curve $\gamma$ of $\Si$, define the endomorphism $Z_k ( \gamma)$ of $V_k ( \Si)$ by  
$$Z_k ( \ga) := Z_k(\Si\times [0,1],\gamma\times[1/3,2/3],2,0) \in \operatorname{End} V_k ( \Si) .$$ In particular, with $\mu = \phi^{-1} ( S^1 \times \{1 \})$ and $\la = \phi^{-1} ( \{1 \} \times S^1 )$ we obtain two operators 
$$Z_k (\mu), Z_k ( \la) \in \End(V_k ( \Si)) .$$
Finally for any $g \in \sldeux$, consider the mapping cylinder $M_g = S^1 \times S^1  \times [0,1]$ where the boundary $- (S^1 \times S^1) \cup   S^1 \times S^1$ is identified with $- \Si \cup \Si$ through $\phi \cup \phi \circ g$.  We set  
$$ Z_k (g) := Z_k ( M_g, \emptyset, \emptyset, 0 ) \in \operatorname{End} (V_k( \Si) ) $$
Consider the set of generators of $\sldeux$ 
\begin{gather}  \label{eq:def_T_S}
T=\begin{pmatrix} 1&1\\ 0&1\end{pmatrix}, \qquad S=\begin{pmatrix} 0&-1\\ 1&0\end{pmatrix}.
\end{gather}

\begin{theo}[\cite{bhmv}]  \label{basetqft}
The family $(e_\ell)_{ \ell =1 , \ldots , k-1}$ is an orthonormal basis of $V_k( \Si)$. For any $\ell = 1, \ldots , k-1$,
$$Z_k(\mu)e_\ell=-2\cos \Bigl( \frac{\pi \ell}{k} \Bigr) e_\ell,\qquad Z_k(\la)e_\ell=- (e_{\ell-1}+ e_{\ell+1})$$
where $e_0=e_k=0$. Furthermore there exists two integers $n$ and $n'$ which does not depend on $k$ such that 
$$Z_k(T)e_\ell= \tau_k^{n}  e^{\frac{i\pi(\ell^2-1)}{2k}}e_\ell,\qquad Z_k(S)e_\ell= \tau_k^{n'} \sqrt{\frac{2}{k}}\sum_{\ell'=0}^{k-1}\sin \Bigl( \frac{\pi \ell\ell'}{k} \Bigr) e_{\ell'}$$
where $\tau_k$ is given in (\ref{eq:def_tauk}).
\end{theo}

To compare later with the geometric construction, observe that the spaces $V_k( \Si, \nu)$ do not depend on $\nu$ in the following sense: for any two Lagrangians $\nu_0$ and $\nu_1$, we have a family of isomorphisms 
$$  Z_k(\Sigma\times[0,1],\emptyset,\emptyset,n) : V_k( \Si, \nu_0) \rightarrow V_k( \Si, \nu_1) , \qquad n \in \Z$$
We could consider only $n=0$ but the identifications between three spaces $V_k( \Si, \nu_i), i =0,1,2$, defined in this way would not be mutually compatible because of Equation (\ref{eq:comp_maslov}).

\subsection{Geometric quantization of the torus}  \label{sec:torus-quantization}

The geometric quantization of the torus can be made very explicit by using Theta series and representation of the finite Heisenberg group. A good reference for this material is \cite{mumford}, chapter I.3. Here we follow closely \cite{l2}. The only original point in our treatment is the use of half-form lines.

Consider a real two dimensional symplectic vector space $(E, \om)$. Let $\al \in
\Om^1 (E, \C)$ be given by $\al_x(y) = \frac{1}{2} \om (x,y)$. Denote by $L$
the trivial complex line bundle over $E$ endowed with the connection
$d + \frac{1}{i} \al$.  

Let $k$ be any positive integer called the
level. Choose a linear complex structure $j$ of $E$ compatible with $\om$. 
Then the $k$-th power
of $L$ has a unique holomorphic structure compatible with the
connection. The holomorphic sections of $L^k$ are the sections satisfying the
Cauchy-Riemann equation 
$$ Z.\Psi + \frac{k}{i} \al (Z) \Psi = 0$$
for any anti-holomorphic vector $Z$ of $E$. 

Let $K_j$ be the canonical line of $(E,j)$, $K_j = \{ \al \in E^* \otimes \C / \al ( j \cdot) = i \al \}$.  Choose a half-form line of $(E, j)$, that is a complex line $\delta$ together with an isomorphism $\varphi : \delta ^{\otimes 2} \rightarrow K_j$. The line $K_j$ has a natural scalar product such that the square of the norm of $\al$ is $i \al \wedge \bar{\al} / \om$. We endow $\delta$ with the scalar product $\langle \cdot, \cdot \rangle_\delta$ making $\varphi$ an isometry. In the following we denote also by $\delta$ the trivial holomorphic Hermitian line bundle with base $E$ and fiber $\delta$.

The Heisenberg group at level $k$ is $E \times U(1)$ with
the product
\begin{gather} \label{eq:prod_Heisenberg} 
 (x,u).(y,v) = \Bigl( x+y, uv \exp \Bigl( \frac{ik}{2} \om (x,y)\Bigr)\Bigr).
\end{gather}
The same formula where $(y,v) \in L^k$ defines an action of the Heisenberg group on $L^k$. We let the Heisenberg group act trivially on $\delta$. For any $x\in E$, we denote by $T_x^*$ the pull-back by the action of $(x, 1)$, explicitly 
\begin{gather} \label{eq:pull_back}
 (T_x^* \Psi )(y) = \exp \Bigl( - i\frac{k}{2} \om (x, y) \Bigr) \Psi ( x+ y)
\end{gather}
for any section $\Psi$ of  $L^k \otimes \delta$.

Let $R$ be a lattice of $E$ with volume $4\pi$. Then $ R \times \{ 1
 \}$ is a subgroup of the Heisenberg group at level $k$. The space $\Hilb_k(j, \delta) $
 of $R$-invariant holomorphic section of $L^k \otimes \delta$ has dimension $2k$. It
 has to be considered has the space of holomorphic sections of the line
 bundle $ L^k \otimes \delta/R\times \{ 1\}$ over the torus $  E/R$. We define the scalar product of
 $\Psi_1, \Psi_2 \in \Hilb_k (j, \delta)$ by 
\begin{gather} \label{eq:scalar_product}
 \langle \Psi_1 , \Psi_2 \rangle  = \int_{ D} \langle \Psi_1 (x), \Psi_2 (x) \rangle_{\delta} \; |\om | (x)
\end{gather}
where $D$ is any fundamental domain of $R$.

The commutator subgroup of $R\times \{1 \}$ in the Heisenberg group at level $k$
is $\frac{1}{2k} R \times U(1)$. It acts on $\Hilb_k(j, \delta)$. This representation has canonical basis. 

\begin{theo} \label{sec:rep_Heisenberg}
Let $(\mu, \la)$ be a basis of $R$ such that $\om (\mu , \lambda) = 4 \pi$. Let $\Om_\mu$ be a vector of $\delta$ such that $\varphi(\Om_\mu^{\otimes 2}) (\mu) =1$. Then there exists a unique orthonormal basis $(\Psi_{\ell})_{\ell \in \Z / 2k \Z }$ of $\Hilb_k(j, \delta)$ such that
\begin{gather} \label{eq:rep_heis}
 T_{\mu/2k}^* \Psi_\ell = e^ {i \ell  \frac{\pi}{k}} \Psi _\ell , \qquad T^*_{\lambda / 2 k }
\Psi_\ell = \Psi_{\ell +1 }.
\end{gather}
and
\begin{gather} \label{eq:normalisation}
\Psi_0 ( 0 ) = \Bigl( \frac{k}{2 \pi} \Bigr) ^{1/4} \Om_\mu \sum_{n \in \Z} e^{2 i \pi k n^2 \tau}
\end{gather} 
where  $\tau = \al + i \be$, with $\al$ and $\be$ determined by the condition $\lambda = \alpha \mu + \beta j \mu$.
\end{theo}

The result is standard, except for the normalization with half-form. Observe that $\Om_\mu$ is determined up to a sign by the condition $\varphi(\Om_\mu^{\otimes 2}) (\mu) =1$. So for any positively oriented basis $( \mu, \la)$ of $R$, there exists exactly two basis satisfying the conditions of Theorem \ref{sec:rep_Heisenberg}. 
\begin{proof} 
The vectors $\Psi_\ell$ may be explicitly computed with theta series as follows. Let $p,q: E \rightarrow \R$ be the linear coordinates dual to $\mu, \lambda$. Define $\tau$ as in the statement. Then $p + \tau q$ is a holomorphic coordinate of $(E,j)$. Furthermore 
\begin{gather} \label{eq:def_t} 
t := \exp( 2 i \pi (p + \tau q)q)
\end{gather}
is a holomorphic section of $L$. It verifies $T^*_{\mu /  2k} t^k = t^k$. Consider the series
\begin{xalignat}{2} \label{eq:ground_state}
 \Psi_0 = & \Bigl( \frac{k}{2 \pi} \Bigr)^{1/4}  \Om_\mu \otimes \sum_{n \in \Z} T^*_{n \lambda } t^k  \\ = & \Bigl( \frac{k}{2 \pi} \Bigr)^{1/4}  \Theta(p + \tau q, \tau) \;\Om_\mu \otimes t^k  \label{eq:ground_state_2}
\end{xalignat}
where $\Theta$ is the theta function 
\begin{gather} \label{eq:theta}
\Theta (z, \tau) = \sum_{n \in \Z} e^{ 4 i \pi  z   k n + 2 i \pi k n^2 \tau}.
\end{gather}
This series converges because $\tau$ has a positive imaginary part.  Indeed $\om ( \mu , \la ) = 4 \pi$ implies that $\be = 4 \pi / \om ( \mu, j \mu)$ which is positive.
Using again that $\om ( \mu, \lambda) = 4 \pi$, we deduce from (\ref{eq:pull_back}) that 
\begin{gather} \label{eq:commutation} T^*_{\mu /2k} T^*_{\lambda/2k} = e^{i \pi /k} T^*_{\lambda /2k} T^*_{\mu /2k}.
\end{gather}
Hence $\Psi_0$ belongs to $\Hilb_k(j, \delta)$ and satisfies $T^*_{\mu/2k} \Psi_0 = \Psi_0$. Furthermore $\Psi_0$ has norm 1. Indeed, $\|\Om_\mu ^2 \|^2 = \be/2 \pi$ and a standard computation shows that the norm of $\Theta ( p + \tau q, \tau) t^k$ is $(2 \pi ) ^{1/2} k^{-1/4} \beta ^{- 1 /4} $ (cf. as instance Theorem 5.1 of \cite{l2}).  
 Once $\Psi_0$ is known, the $\Psi_\ell$ for $\ell \neq 0$ are determined by the second equation of (\ref{eq:rep_heis}) and the first equation of (\ref{eq:rep_heis}) follows from (\ref{eq:commutation}). That the dimension of $\Hilb_k( j , \delta)$ is $2k$ is a well-known fact, so the $\Psi_\ell$'s form a basis.  
\end{proof}

We can identify the various Hilbert spaces $\Hilb_k(j, \delta)$ obtained by varying the complex structure and half-form line as follows. Denote by $\Psi_0 ( \mu,\lambda, j, \delta)$ the vector in the line $\Hilb_k ( j, \delta) \cap \ker ( T^*_{\mu/ 2k} - \id)$ satisfying (\ref{eq:normalisation}). $\Om_\mu$ being only defined up to a sign, the same holds for  $\Psi_0 ( \mu,\lambda, j, \delta)$. 
 Given two pairs $(j_1, \delta_1)$ and $(j_2 , \delta_2)$, there exists by Theorem  \ref{sec:rep_Heisenberg} a unitary map $U$  from $ \Hilb_k(j_1, \delta_1)$ to $\Hilb_k(j_2, \delta_2)$ unique up to a phase which is a morphism of the Heisenberg group representation. We can assume that for some basis $(\mu, \lambda)$ of $R$, $$U(\Psi_0 ( \mu, \lambda ,j_1, \delta_1)  ) = \pm \Psi_0 ( \mu,\lambda, j_2, \delta_2).$$ Such a $U$ is unique up to a sign. The important point is that the unitary map defined by this condition does not depend on the choice of the basis $(\mu, \lambda)$.

\begin{theo} \label{sec:geo_isomorphism}
For any positively oriented basis $(\mu' , \lambda')$ of $R$,  $$U(\Psi_0 ( \mu', \lambda', j_1, \delta_1)  ) = \pm \Psi_0 ( \mu', \lambda', j_2, \delta_2) .$$  
\end{theo}

\begin{proof} Since $\sldeux$ is generated by the matrices $T$ and $S$ defined in (\ref{eq:def_T_S}), it is sufficient to prove the result with $(\mu',\la')$ equal to $(\mu , \mu + \la)$ or $(\la, - \mu)$. Observe that the  parameter $\tau$ depends on $j$, $\mu$ and $\la$. We have that $$\tau ( j, \mu, \mu + \la) = \tau ( j, \mu,\la) +1 .$$
The theta function (\ref{eq:theta}) satisfies $\Theta ( 0, \tau +1 ) = \Theta ( 0, \tau)$. It follows that
\begin{gather} \label{eq:trans_T}
 \Psi_0 ( \mu, \mu + \la, j , \delta ) = \pm \Psi_0 ( \mu, \la, j , \delta )
\end{gather}
which shows the result for $(\mu' , \la') = ( \mu , \mu + \la)$.  

Let us prove that 
\begin{gather} \label{eq:trans_S}
 \Psi_0 ( \la, -\mu , j , \delta ) = \pm e^{-i \pi /4} (2k ) ^{-1/2} \sum_{\ell \in \Z/ 2 k \Z} T^*_{\ell \lambda /2k} \Psi_0 ( \mu, \la, j , \delta )
\end{gather}
By (\ref{eq:rep_heis}), $\sum_{\ell \in \Z/ 2 k \Z} T^*_{\ell \lambda /2k} \Psi_0 ( \mu, \la, j , \delta )$ is an eigenstate of $T^* _{\la /2 k}$ with eigenvalue 1. By (\ref{eq:def_t}) and (\ref{eq:pull_back}), its value at $0$ is:
\begin{gather*} 
\Bigl( \frac{k}{2 \pi} \Bigr)^{1/4}  \Om_\mu \sum_{n \in \Z} \exp \Bigl( \frac{i \pi \tau}{2 k } n^2 \Bigr) 
\end{gather*} 
with $\tau = \tau ( j , \mu, \la)$. 
To conclude we use that $\Om_\mu^2 = \tau \Om_\lambda^2$,  $\tau ( j, \la,  - \mu ) = -1/ \tau $ and the following identity 
$$ (2k ) ^{-\frac{1}{2}} \sum_{n \in \Z} \exp{\Bigl(  \frac{i\pi \tau}{2k} n^2 \Bigr) } = \Bigl( \frac{i}{\tau} \Bigr)^{1/2} \sum_{n \in \Z} \exp \Bigl( {- \frac{2i \pi k}{\tau} n^2} \Bigr). $$
which follows from Poisson summation formula. So 
$$     \Om_\mu \sum_{n \in \Z} \exp \Bigl( \frac{i \pi \tau}{2 k } n^2 \Bigr) = \pm (2k)^{\frac{1}{2}}e^{i \pi /4}    \Theta ( 0, - 1 /\tau)  \Om_{\la}
$$
and (\ref{eq:trans_S}) follows. Using that $U$ is a morphism of Heisenberg group representations, we obtain the final result for $(\mu' , \la') = (\la,  - \mu )$.    
\end{proof}

In the sequel we consider the subspace $\Hilb_k^{\alt} ( j, \delta)$ of $\Hilb_k (j , \delta)$ consisting of the sections $\Psi$ satisfying
$$ \Psi ( -x) = - \Psi (x ), \qquad \forall x \in E. $$
We call these sections the alternating sections. Introduce as previously a base $( \mu, \la)$ of $R$ and the associated orthonormal basis $( \Psi_\ell)$. Then the family 
$$ \frac{1}{\sqrt 2} \bigl( \Psi_{\ell} - \Psi_{-\ell} \bigr), \qquad \ell = 1, \ldots , k-1 $$ 
is an orthonormal  basis of $\Hilb_k ^{\alt} ( j , \delta)$. Observe furthermore that the unitary map $U$ introduced above restricts to a unitary map between $\Hilb_k^{\alt}  (j_1, \delta_1)$ and $\Hilb_k ^{\alt} ( j_2, \delta_2)$. 

\subsection{Equivalence between the geometric and topological constructions} \label{sec:equiv}

Let $\Sigma$ be a closed surface with genus 1. On one hand, choosing a Lagrangian subspace $\nu$ in $H^1( \Si)$ we define the Hermitian space $V_k ( \Sigma, \nu)$. On the other hand, consider the vector space $E = H_1 ( \Si, \R)$ endowed with the symplectic product 
\begin{gather} \label{eq:prod-symp-E} 
 \om (x,y) = 4 \pi( x \cdot y) 
\end{gather}
where the dot stands for the intersection product.  Then $R = H_1 ( \Si, \Z)$ is a lattice of $E$ with volume $4 \pi$.  Introduce a linear complex structure $j$ on $E$ together with a half-form line $\delta$ and define $\Hilb_k ( j, \delta)$ as in Section \ref{sec:torus-quantization}. 

The spaces $V_k ( \Sigma, \nu)$ and $\Hilb_k ^{\alt}( j, \delta)$ have the same dimension. Let us define a preferred class of isomorphisms between them. Let $\phi$ be a diffeomorphism between $\Sigma$ and the standard torus $S^1 \times S^1$. Then on the topological side, by Theorem  \ref{basetqft}, we have an orthonormal basis $(e_\ell)_{\ell=1, \ldots , k-1}$ of $V_k ( \Sigma, \nu)$. 

On the geometric side, let $ \mu$ and $\la$ be the homology classes of  $\phi^{-1}( S^1 \times \{ 1 \})$ and $\phi^{-1}( \{ 1 \} \times S^1)$ respectively, so that $(\mu , \la )$ is a positively oriented basis of $R$. Choose one of the two associated basis $(\Psi_\ell)$ of $\Hilb_k ( j, \delta)$ given in  Theorem \ref{sec:rep_Heisenberg}. We define the isomorphisms $I_k$ from $V_k ( \Si, \nu)$ to $\Hilb_k ^{\alt}( j, \delta)$ by 
$$ I_k(e_\ell) = \frac{1}{\sqrt 2} \bigl( \Psi_\ell - \Psi_{-\ell} \bigr)$$ for $\ell = 1, \ldots, k-1$. 

\begin{theo} \label{sec:equiv-top-geom}
Let $\phi$ and $\phi'$ be two oriented diffeomorphisms $\Si \rightarrow S^1 \times S^1$. Denote by $I_k, I_k' : V_k ( \Si, \nu) \rightarrow \Hilb^{\alt}_k ( j , \delta)$, $k \in \Z_{>0}$ the associated isomorphisms. Then there exists two integers $n$ and $n'$ such that for any $k$, we have
$$  I_k'  =  e^{i\pi (\frac{n}{4}+\frac{n'}{2k})} I_k .$$
\end{theo}

\begin{proof} Since any oriented diffeomorphism of the torus $S^1 \times S^1$ is isotopic to a linear map and  $\sldeux$ is generated by $S$ and $T$, we can assume that $ \phi' \circ \phi^{-1} =g$ with $g=T$ or $S$. Let $(e_\ell)$ and $(e'_\ell)$ be the two basis of $V_k ( \Si, \nu)$ defined from $\phi$ and $\phi'$ respectively. Then  
$$e_\ell' = \tau_k^m Z_k (g) e_\ell , \qquad \ell =1 , \ldots , k-1 $$ 
where $Z_k (g)$ is the morphism associated to the mapping cylinder of $g$ as in Section \ref{tqft}. Here $m$ is an integer independent of $k$ coming from the composition rule (\ref{eq:comp_maslov}). The endomorphisms $Z_k(T)$ and $Z_k(S)$ are given explicitly in Theorem \ref{basetqft}. 

Let $(\mu, \la)$ and $(\mu' , \la')$ be the two basis of $R$ defined from $\phi$ and $\phi'$ respectively. If $g= T$, then  $(\mu' , \la')= (\mu + \la , \la)$. So by equation (\ref{eq:trans_T}) in the proof of Theorem  \ref{sec:geo_isomorphism}, we have that $ \Psi_0' = \varepsilon \Psi_0 $ with $\varepsilon \in \{ 1, -1 \} $. Consequently
\begin{gather*} \Psi_\ell'  =    (T^*_{(\la + \mu )/2k  })^\ell \Psi_0'  =  \varepsilon e^ {\frac{i \pi}{2 k} \ell^2 }  \Psi_\ell
\end{gather*}
and comparing with the formula of Theorem \ref{basetqft}, we conclude. For $g=S$, we have $(\mu' , \la') = (\la, -\mu )$. So we deduce from (\ref{eq:trans_S}) 
\begin{gather*}  \Psi_{\ell} ' = (T^*_{-\mu /2k })^\ell \Psi_0'  =  \ep e^{- i \pi /4} \frac{1}{\sqrt{2k}} \sum_{\ell' \in \Z / 2k \Z} e^{- \frac{i \pi}{k} \ell  \ell'}  \Psi_{\ell'}
\end{gather*}
with $\ep \in \{1 , -1\}$, so
$$ ( \Psi'_\ell - \Psi_{-\ell}')  =  \ep i  e^{- i \pi /4} \sqrt{\frac{2}{k}} \sum_{\ell'=1}^{k-1} \sin \Bigl( \frac{\pi\ell\ell'}{k} \Bigr) \bigl( \Psi_{\ell'} - \Psi_{-\ell'} \bigr) $$ 
This formula agrees with the formula for $Z_k(S)e_\ell$, which ends the proof. 

\end{proof}

For any curve $\gamma$ of $\Si$, we defined in Section \ref{tqft} an endomorphism $Z_k ( \gamma)$ on $V_k ( \Si)$. Orient $\ga$ and define
$$T_k( \gamma) = - (T_{\ga / 2k}^* + T_{ -\gamma/2k}^*)$$ 
which  preserves the subspace of alternating sections of $\Hilb_k$. 

\begin{prop} 
For any curve $\gamma$ of $\Si$, the isomorphisms $I_k$ intertwine $Z_k ( \ga)$ and $T_k( \gamma).$
\end{prop} 

\begin{proof} 
By Theorem \ref{sec:equiv-top-geom}, we may choose a diffeomorphism $\phi:\Si\to S^1\times S^1$ such that $\varphi(\gamma)= S^1 \times \{ 1\} $ and use it to construct the isomorphisms $I_k$. We are reduced to compare $Z_k(\ga )$ and $T_k(\gamma)$.
  We have by Formula (\ref{eq:rep_heis}):
$$  T_k(\gamma)  ( \Psi_\ell - \Psi_{- \ell} ) = - 2\cos \Bigl(  \frac{\ell \pi}{k} \Bigr) ( \Psi_\ell - \Psi_{- \ell} ).$$ 
We recover the action of $Z_k ( \gamma)$ on the basis $(e_\ell)$.
\end{proof}

\section{Asymptotic properties of curve operators and basis vectors}\label{sec:asymptotic_curve_and_basis}

\subsection{Curve operators as Toeplitz operators} \label{sec:tore_Toeplitz_operators}

Consider as in Section \ref{sec:torus-quantization} a 2-dimensional symplectic vector space $(E, \om)$ with a lattice $R$ of volume $4 \pi$. Introduce a linear complex structure $j$ and a half-form line $\delta$ and let $\Hilb_k = \Hilb_{k} ( j,\delta)$. 
Let $\Hilb_k^2$ be the space of $R$-invariant sections of $L^k \otimes \delta$ which are locally of class $L^2$. This is a Hilbert space where the scalar product is defined by (\ref{eq:scalar_product}). Furthermore $\Hilb_k$ is a (closed) finite dimensional subspace of $\Hilb_k^2$. We let $\Pi_k$ be the orthogonal projector of $\Hilb^2_k$ onto $\Hilb_k$. If $f$ is a $R$-invariant function on $E$, we denote by $M(f)$ the multiplication operator by $f$ of $\Hilb_k^2$. 

A {\em Toeplitz operator} is a family $(T_k \in \End ( \Hilb_k), \; k \in \Z_{>0})$ of the form
$$ T_k =  \Pi_k M(f(\cdot, k ) ) + R_k : \Hilb_k \rightarrow \Hilb_k , \qquad k=1,2, \ldots $$
where 
\begin{itemize} 
\item[-] $(f( \cdot, k))_k$ is a sequence of $\Ci_R ( E )$ which admits an asymptotic expansion of the form $f_0 + k^{-1} f_1 + \cdots $ for the $\Ci$ topology, with coefficients $f_0, f_1, \ldots \in \Ci_R ( E)$. 
\item[-] the family $(R_k \in \End ( \Hilb_k), \; k \in \Z_{>0}) $ is a $O(k^{-\infty})$, i.e. for any $N$, there exists a positive $C_N$ such that for every $k$, $ \| R_k \| \leqslant C_N k^{-N}$. Here $\|\cdot \|$ is the uniform norm of operators. 
\end{itemize}
As a result, the coefficients $f_0$, $f_1$, \ldots are uniquely determined by the family $(T_k)_k$. We call the formal series $f_0 + \hbar f_1 + \hbar^2 f_2 +  \cdots $ the total symbol of $(T_k)$, $f_0$ the principal symbol of $(T_k)$ and $f_1 -\frac{1}{2} \Delta f_0$ the subprincipal symbol of $(T_k)$. As another basic result, the product $(S_k T_k)_k$ of two Toeplitz operators $(S_k)$ and $( T_k)$ is a Toeplitz operator whose principal and subprincipal symbol are 
$$ f_0 g_0 , \quad f_1 g_0 + f_0 g_1 + \frac{1}{2i} \{ f_0 ,  g_0 \} $$
where $f_0$, $f_1$ (resp. $g_0$, $g_1$) are the principal and subprincipal symbols of $(S_k)$ (resp. $(T_k)$).

\begin{theo} \label{sec:Toeplitz_operators}
For any vector $\nu \in R$, the sequence $(T^*_{\nu/k}: \Hilb_k \rightarrow \Hilb_k)_k$ is a Toeplitz operator with principal symbol $\si \in \Ci_R (E)$ given by $$\si( x)  = \exp ( i \om (x, \nu))$$ and vanishing subprincipal symbol.  
\end{theo}

The result is certainly well-known to specialists. We provide a proof in the appendix. As a consequence the family of curve operators  $(T_k(\gamma): \Hilb_k \rightarrow \Hilb_k)_k$ is a Toeplitz operator with principal symbol 
$$ x \rightarrow - 2 \cos ( \om ( x, \ga /2)) $$
and vanishing subprincipal symbol.

\subsection{Basis vectors as Lagrangian states} \label{sec:lagrangian-states}

Consider a positively oriented basis $(\mu, \la )$ of $R$ and choose $\Om_\mu \in \delta$ such that   $\varphi(\Om_\mu^{\otimes 2}) (\mu) =1$. Then for any positive integer $k$, we have an orthonormal basis $(\Psi_\ell)$ of $\Hilb_k$ by Theorem \ref{sec:rep_Heisenberg}. As $k$ tends to infinity, these vectors concentrate on the fibers of the affine fibration of $E/R$ directed by $\mu$. More precisely, $\Psi_{\ell} $ is a Lagrangian state supported by the circle 
$$\{  x \in E ; \; \om(x, \mu) = 2 \pi \ell/ k \}/ (\mu \Z) =  \Bigl( -\frac{\ell }{2k} \lambda + \R \mu \Bigr) /  \mu \Z $$ 
This follows from general properties of Toeplitz operators because each $\Psi_\ell$ is an eigenstate of $T^*_{\mu/2k}$ with eigenvalue $\exp(i \ell \pi/k )$ and the symbol of $T^*_{\mu/2k}$ is $\exp ( i \om ( x,\mu)/2)$. But it is actually simpler to check this property directly from the construction of the $\Psi_\ell$ as we do in the following proposition.

\begin{prop} \label{sec:base_lagrangian_states}
For any $\delta \in (0,1)$, there exists a positive $C$ such that for any positive integer $k$  we have 
\begin{itemize} 
\item for all  $x \in \frac{1}{2}\la + \frac{1}{2}[-\delta, \delta ]\la + \R \mu$ modulo $R$,
$$\bigl| \Psi_0 ( x)  \bigl| \leqslant C e^{-k/C} .$$
\item for all  $x \in  [- \delta , \delta] \la  + \R \mu$, 
$$ \bigl| \Psi_0 ( x) - \Bigl( \frac{k}{2\pi} \Bigr)^{1/4} t^k (x) \otimes \Om_{\mu}  \bigl| \leqslant C e^{-k/C} $$ 
where $t$ is the holomorphic section of $L$ whose restriction to the line $ \R \mu $ is constant equal to 1.
\end{itemize}
\end{prop}

\begin{proof} 
We start from equation (\ref{eq:ground_state}) defining  $\Psi_0$. The section $t$ given in (\ref{eq:def_t}) is holomorphic and its restriction to $\mu \R = \{ q =0 \}$ is constant equal to $1$. It is characterized by these conditions. By (\ref{eq:def_t}), we have  
$$   \bigl| t (x) \bigr| = e^{ -2 \pi \be q^2(x)   }, \qquad \forall x \in E $$
where $\be>0$ is the imaginary part of $\tau$ and $q$ the linear coordinate of $E$ such that $q(\la) =1$ and $q( \mu ) =0$.  We have 
$$  \bigl| T^*_{n\lambda} t (x) \bigr| =  e^{ -2 \pi  \be [q(x) +n ]^2 }$$
To conclude it suffices to prove that there exists $C$ such that for any real $s$ and any positive $k$, one has  
$$ |s | \leqslant \delta \quad \Rightarrow \quad \sum_{ n \in \Z, \; n \neq 0 } e ^{-k d  ( s - n ) ^2} \leqslant C e^{-k/C} $$
and 
$$  |s| \leqslant \frac{\delta}{2}  \quad \Rightarrow \quad \sum_{ n \in \Z  } e ^{-k d  ( s+\frac{1}{2}  - n ) ^2} \leqslant C e^{-k/C} $$
with $d = 2 \pi \beta$. 
\end{proof}

Using that $\Psi_{\ell} = T^*_{\ell \la / 2k} \Psi_0$, we deduce a uniform description of the $\Psi_\ell$'s. For instance, for any $\ell$, we have   
$$ \bigl| \Psi_\ell ( x) - \Bigl( \frac{k}{2\pi} \Bigr)^{1/4} T^*_{\ell \la / 2k } t^k (x) \otimes \Om_{\mu}  \bigl| \leqslant C e^{-k/C} $$ 
 for all $ x \in -\frac{\ell }{2k} \lambda + [- \delta, \delta] \lambda + \R \mu$. Furthermore the holomorphic section $T^*_{\ell \la / 2k } t^k$ is determined up to a phase by the condition that its restriction to $-\frac{\ell }{2k} \lambda + \R \mu$ is flat with a unitary pointwise norm equal to 1.

\subsection{Solid torus state}  \label{sec:solid-torus-state}

As a first application, we can describe the asymptotic behavior of 
$$ Z_k ( N) := Z_k ( N, \emptyset , \emptyset , 0) \in V_k (\Si, \nu ) $$ 
where $N$ is a solid torus with boundary identified with $\Si$. Denote by  $F$ the kernel of the natural map $ H_1(\Si) \rightarrow H_1 (N)$. As in Section \ref{sec:equiv}, we identify the space $V_k( \Si, \nu)$ with $\Hilb_k^{\alt} \subset \Hilb_k$, where $\Hilb_k$ is obtained by quantizing the torus $E/R$ with $E = H_1(\Si, \R)$ and $R =  H_1 ( \Si, \Z)$.
\begin{theo} \label{theo:solid-torus-state}
For any $x \in E \setminus ( F + R)$, there exists a neighborhood $U $ of $x$ in $E$ and a positive $C$ such that for any integer $k>0$ and for all $y \in U$  
$$ | Z_k(N) (y) | \leqslant C e^{-k/C} $$
 Furthermore there exists a neighborhood $V$ of $F$ and a positive $C$ such that for any $y \in V$ and positive $k$, 
$$ \Bigl| Z_k (N) (y) -   e^{\frac{i \pi}{4} m} \theta_k  \Bigl( \frac{k}{2 \pi} \Bigr)^{1/4} t^k(y) \otimes \sigma(y)  \Bigr| \leqslant C e^{-C/k} $$
where $m$ is an integer,  $t$ is the holomorphic section of $L$ which is equal to $1$ on $F$ and  $\si$ is the holomorphic section of $\delta$ whose restriction to $F$ is 
given by $$ \si ( s \gamma ) = \sqrt 2 \sin ( 2 \pi s ) \Om_{\gamma}, \qquad s \in \R $$ 
with $\gamma $ a generator of $R \cap F$.  Finally, $\theta_k $ is a complex sequence admitting an asymptotic expansion of the form $1+ a_1 k^{-1} + a_2 k^{-2} + \ldots$. 
\end{theo}

\begin{proof} 
By Theorem \ref{sec:equiv-top-geom}, one can choose the diffeomorphism $\phi : \Si \rightarrow S^1 \times S^1$ entering in the definition of the isomorphisms $V_k ( \Si, \la) \rightarrow \Hilb_k^{\alt}$ in such a way that $\phi^{-1}( S^1 \times \{1\})$ is a generator of $F\cap R $.  Then introducing as previously the basis $(\Psi_\ell)$ associated to $( \mu, \la)$, we have  
$$Z_k(N) = \frac{1}{\sqrt {2}} ( \Psi_1 - \Psi_{-1})  .$$ 
By (\ref{eq:pull_back}) and (\ref{eq:def_t}), we have
$$ T^*_{\ell \la /2k } t^k  = e^{ \frac{i \pi }{2k} \tau \ell ^2 + 2 i \pi \ell ( p + \tau  q)  } t^k $$ 
with $p$, $q$ the linear coordinates of $E$ associated to $\mu$, $\la$. So by Proposition \ref{sec:base_lagrangian_states} we have on a neighborhood of $F$ the following inequality
$$ \Bigl|  (\Psi_{1}  - \Psi_{-1}) - \Bigl( \frac{k}{2 \pi} \Bigr)^{1/4} t^k \otimes \sigma(k) \Bigr| \leqslant C k^{-C/k} $$
where $\sigma(k)= 2i e^{ \frac{i \pi}{2k } \tau} \sin ( 2 \pi ( p + \tau q)) \Om_\la.$ 
\end{proof}

\section{Knot state}

\subsection{Jones polynomials}\label{sec:jones}

Let $M$ be a compact oriented 3-manifold. We will call banded link in $M$ the embedding of a finite disjoint union of copies of $S^1\times [0,1]$. We denote by $S(M,t)$ the $\C[t^{\pm 1}]$-module generated by isotopy classes of banded links modulo the Kauffman relations shown if Figure \ref{fig:kauffman}.

\ifpdf

\begin{figure}[width=8cm,height=3cm]
\begin{center}
\begin{pspicture}(-2,0)(3,3)
\includegraphics{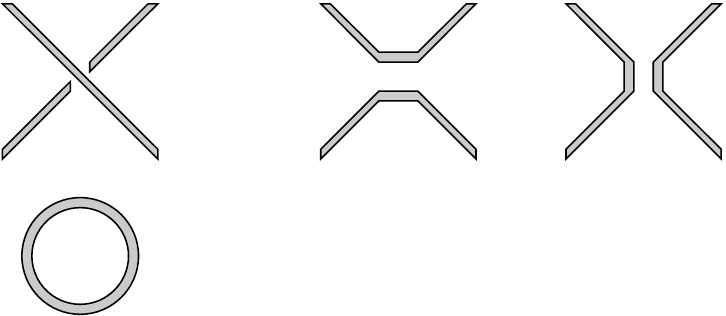}
\put(-5.2,2.3){$=A$}
\rput(-1.9,2.5){$+A^{-1}$}
\rput(-4,0.7){$=(-A^2-A^{-2})\quad \emptyset$}
\end{pspicture}
\caption{Kauffman bracket skein relations}\label{fig:kauffman}
\end{center}
\end{figure}

\fi

It is a consequence of Reidemeister theorem that $S(S^3,t)\simeq \C[t^{\pm 1}]$ via a map
$\langle\cdot\rangle$ called Kauffman Bracket, normalized in such a
way that the bracket of the empty knot is 1. For a banded knot $K$ in $S^3$, $\langle K\rangle$ is a version of the Jones polynomial.
More generally, consider the sequence of polynomials defined by the recursion relation 
$$T_0=0\quad T_1=1 \quad T_{\ell+1}+xT_\ell+T_{\ell-1}=0.$$
Given a banded knot $K$ in a 3-manifold $M$, one can cut the annulus into $\ell$ parallel strands. We denote this banded link by $K^\ell$. For any polynomial $P=\sum_i a_i x^i$ we set $P(K)=\sum_i a_i K^i \in S(M,t)$. Using this construction, we define the $\ell$-th colored Jones polynomial of a knot in $S^3$ by 
$$ J_\ell^K=\langle T_\ell(K)\rangle \in\C[t^{\pm 1}].$$
This sequence satisfies $J_{-\ell}^K=-J_\ell^K$ and for any positive integer $k$ and $\ell \in \Z$
$$ J_{\ell+2k}^K( -e^{i \pi /2k} ) =  J_{\ell}^K( -e^{i \pi /2k} )$$
There are many distinct normalizations and this one is not the most familiar. The most standard normalization $\tilde{J}_\ell^K$ is related to this one by the formula
\begin{equation}\label{eq:normalization}
J_\ell^K(t)=\tilde{J}_\ell^K(t^{-4})\frac{t^{2\ell}-t^{-2\ell}}{t^2-t^{-2}}.
\end{equation}
An exhaustive list of values of $\tilde{J}_\ell^K$ for small $\ell$ and small $K$ is available in the Knot Atlas at {\tt http://katlas.math.toronto.edu/wiki/Main\_Page}. Unfortunately, except for very few knots $K$, we do not have a closed formula for $J_\ell^K$ taking into account all values of $\ell$. 

The Jones polynomials are relevant in this article because of the formula
\begin{gather} \label{eq:jones}
Z_k(S^3,K,\ell,0)=\sqrt{\frac{2}{k}}   \sin \Bigl( \frac{\pi}{k} \Bigr)  J_\ell^K (- e^{i\pi/2k})
\end{gather}
This is merely a definition in the construction of TQFT given in \cite{bhmv}.

\subsection{Knot state} \label{sec:knot-state}

Consider a banded knot $K$ in $S^3$. Let $N_K$ be a closed tubular neighborhood of
$K$, $\Sigma$ be the
oriented boundary of $N_K$ and $E_K$ be the complement of the interior of $N_K$ in $S^3$. 
Let $\mu, \la$ be the isotopy classes
of oriented simple
closed curves in
the peripheral torus $\Si$ such that the linking number of $\la$ and $K$ is zero in $S^3$, $\mu$ bounds a disc in $N_K$ and intersects $\la$ once. 

For any level $k$, since the three manifold $E_K$ bounds the torus $\Si$, it induces a vector 
$$Z_k(E_K) = Z_k (E_K, \emptyset, \emptyset, 0) \in V_k(\Si, \nu ) ,$$ 
called the state of the knot $K$ at level $k$. Here we choose $\nu$ as the line generated by the longitude $\la$. Introduce a diffeomorphism between $\Si$ and $S^ 1 \times S ^1$ such that  $\mu$ and $\la$ are respectively sent to $ S^1\times \{ 1 \}$ and $ \{1\}\times S ^1$. Consider the associated basis $(e_\ell) _{\ell = 1, \ldots, k-1}$ as in Section \ref{tqft}. Then the coefficients in this basis of the knot state are given by the Jones polynomials of $K$:
\begin{gather} \label{eq:ks} 
 Z_k (E_K) =  \sqrt\frac{2}{k} \sin \Bigl( \frac{\pi}{k} \Bigr) \sum_{\ell =1 }^{k-1} J_\ell ^K(-e^{i\pi/2k}) e_\ell  
\end{gather}
This formula is obtained by the following argument. Since the basis $(e_\ell)$ of $V_k(\Si , \nu)$ is orthonormal, the coefficient of $e_\ell$ in $Z_k(E_K)$ is the scalar product $ \langle Z_k(E_K), e_\ell \rangle$. By definition $e_\ell = Z_k ( N_K, K, \ell , 0 )$ where we thickened $K$ in such a way that its two boundary components are unlinked. So 
$$  \langle Z_k(E_K), e_\ell\rangle =  Z_k( S^3,K,\ell,0) $$
because the composition of $(E_K, \emptyset, \emptyset, 0)$  with $ (-N_K,K, \ell , 0 )$ is $(   S^3,K,\ell,0)$. Here we used that the line $\nu$ is the kernel of $H_1( \Si ) \rightarrow H_1( E_K)$ and that the Maslov index in (\ref{eq:comp_maslov}) vanishes when two of the Lagrangians are equal. Finally (\ref{eq:jones}) implies (\ref{eq:ks}). 

As explained in Section \ref{sec:equiv}, we may also consider $Z_k(E_K)$ as a vector in the space quantizing the torus $H_1 (\Si, \R)/ H_1 (\Si,  \Z)$:  
\begin{gather} \label{eq:etat_noeud}
 Z_k(E_K) = \frac{\sin(\pi/k)}{\sqrt{k}}\sum_{\ell \in \Z / 2k
  \Z} J_\ell ^K(-e^{i\pi/2k}) \Psi_\ell 
\end{gather}
where $(\Psi_\ell)$ is the canonical basis associated to $(\mu,\la)$. 
\subsection{Recurrence relations}\label{sec:recurrence-relations}

Let $\sP$ be the set of sequences of polynomials $f:\Z\to \C[t^{\pm 1}]$ satisfying for all $n\in\Z$ and $k\in\Z_{>0}$ the equality $$f_{n+2k}(-e^{i\pi/2k})=f_n(-e^{i\pi/2k}).$$ 
For any sequence $f \in \sP$, we define a family $(Z_k(f)\in \Hilb_k, \; k \in \Z_{>0})$ by the formula 
$$ Z_k(f)=\frac{\sin(\pi/k)}{\sqrt{k}}\sum_{\ell\in\Z/2k\Z} f_\ell(-e^{i\pi/2k})\Psi_\ell.$$
By Formula (\ref{eq:etat_noeud}), for any knot $K$ we have $Z_k(E_K)=Z_k(J^K)$. We can define two operators on $\sP$:
$$(M f)_n=q^n f_n, \qquad (L f)_n = f_{n+1}$$ where we set $q=t^2$. These operators satisfy the equation:
\begin{equation}\label{representation}
Z_k(M f)=T^*_{\mu/2k} Z_k(f) \quad \text{ and } \quad Z_k(L f)=T^*_{-\lambda/2k} Z_k(f).
\end{equation}
The operators $M$ and $L$ generate the quantum torus algebra: 
$$\sT=\C[q^{\pm 1} ] \langle M^{\pm 1},L^{\pm 1}| LM=q ML\rangle .$$ 
We say that a sequence $f$ satisfies a {\em q-differential relation} if there is an element $P\in \sT$ such that  $Pf=0$. More generally, we will call {\em non-homogeneous q-differential relation} an equation of the form $$Pf=R$$ 
for some $P\in \sT$ and $R\in \sP$. It was shown in \cite{gl} that the colored Jones polynomials always satisfy $q$-difference relations. As we will see, the Jones polynomials of the torus knots and the figure eight knot satisfy such equations where $R$ has the form $R_n  = R' (t,t^{n} )$ with $R'$ a rational expression in two variables. 
Thanks to Equation \eqref{representation}, such relations for the Jones polynomial will produce families of linear equations on the corresponding knot state.

\subsubsection{Torus knots}

Let $a,b$ be two relatively prime integers. We will denote by $J_\ell^{a,b}$ the $\ell$-th colored Jones polynomial of the torus knot with parameters $(a,b)$. In particular, $J_\ell^{3,2}$ is the Jones polynomial of the right-handed trefoil. In \cite{morton}, we find the following formula:

$$J^{a,b}_\ell=\frac{t^{ab(1-\ell^2)}}{t^2-t^{-2}}
\sum_{r=(1-\ell)/2}^{(\ell-1)/2}\big(t^{4abr^2-4(a+b)r+2}-t^{4abr^2-4(a-b)r-2}\big).$$
The following proposition comes from a direct computation which was first obtained by Hikami in \cite{hikami}.
\begin{prop} \label{sec:recurrence-relations-torique}
For any pair $a,b$, the colored Jones polynomials of the torus knot with parameters $a,b$ denoted by $J_\ell^{a,b}$ satisfy the following non-homogeneous recurrence equation:
$$J_\ell^{a,b}=P_l^{a,b}+t^{4ab(1-\ell)}J^{a,b}_{\ell-2}$$
where $$P_\ell^{a,b}=\frac{t^{2(ab-a-b)(1-\ell)}}{1-t^{-4}}\big( 1-t^{4a(1-\ell)-4}-t^{4b(1-\ell)-4}+t^{4(a+b)(1-\ell)}\big)$$
\end{prop}

Suppose that we have $a=2$ and $b$ is odd. Then, the colored Jones polynomial of the torus knot with parameters $(2,b)$ can be written as 

$$J_\ell^{2,b}=\frac{q^{b(1-\ell^2)}}{q-q^{-1}}\sum_{r=(1-\ell)/2}^{(\ell-1)/2}\Big( q^{4br^2-2(b+2)r+1}-q^{4br^2-2(2-b)r-1}\Big)$$
Putting $\langle r\rangle=q^{4br^2-2(b+2)r+1}$, we find that the second term in the parenthesis can be written as $\langle r+1/2\rangle$. Hence, we have the formula:
$$J_\ell^{2,b}=\frac{q^{b(1-\ell^2)}}{q-q^{-1}}\sum_{r=1-\ell}^{\ell} (-1)^{\ell+r+1} \langle r/2\rangle.$$

This expression gives the following recurrence relation obtained also by Hikami in \cite{hikami}.
\begin{prop} \label{sec:recurrence-relations-torique2}
For any odd integer $b$, the colored Jones polynomials of the torus knot with parameters $2,b$ satisfy the following non-homogeneous recurrence equation:

$$J_\ell^{2,b}=-q^{b(1-2\ell)}J^{2,b}_{\ell-1}+q^{b(1-\ell)}\frac{q^{2\ell-1}-q^{1-2\ell}}{q-q^{-1}}.$$
\end{prop}

\subsubsection{Figure eight knot}
Let $q=t^2$ and for any integer $n$, set $\{n\}=q^n-q^{-n}$ and $\{n\}!=\prod_{i=1}^n\{i\}$. The following formula for the colored Jones polynomial of the figure eight knot is due to Habiro, see \cite{habiro,masbaum}.

$$J_n^8=\sum_{m=0}^{\infty}\frac{\{n+m\}!}{\{1\}\{n-m-1\}!}.$$
In order to prove various recursion formulas, we define for any polynomial $P\in \C[q,M,x]$ and integer $n$ the bracket 
$$\langle P\rangle_n=\sum_{m=0}^{\infty}\frac{\{n+m\}!}{\{1\}\{n-m-1\}!}P(q,q^n,q^{2m}).$$

\begin{prop}\label{rec1}
For all $P\in \C[q,M,x]$, the following equations hold:
$$\langle P\rangle =  \bigl\langle (M^2+M^{-2}-q^2 x-q^{-2} x^{-1})P(q,M, q^2x) \bigr\rangle +  ( M-M^{-1})R $$
\begin{xalignat*}{2}
L\langle P\rangle = & \bigl\langle (q^3 M^2 x +q^{-3}M^{-2}x^{-1}-q-q^{-1})P(q, M, q^2x) \bigr\rangle \\ &  +     ( qM-q^{-1}M^{-1})R  
\end{xalignat*}
\begin{xalignat*}{2}
L^{-1}\langle P\rangle = & \bigl\langle (q^3 M^{-2} x +q^{-3}M^{2}x^{-1}-q-q^{-1})P(q, M , q^2x) \bigr\rangle \\ & +  (q^{-1}M-qM^{-1}) R 
\end{xalignat*}
where  $R=\frac{P(q,M,1)}{q - q^{-1}} $.
\end{prop}
\begin{proof}
The proof consists in re-indexing the sum. Let us show the first equation.  Replace $m$ by $m+1$ in the definition of $\langle P\rangle_n$. We obtain 
\begin{gather} \label{eq:1}
\langle P\rangle_n=\frac{\{n\}}{\{1\}}P(q,q^n,1)+\sum_{m=0}^{\infty}\frac{\{n+m+1\}!}{\{1\}\{n-m-2\}!}P(q,q^n,q^{2m+2}).
\end{gather}
The first term in the right hand side is $ R (M- M^{-1})$. 
The integrand of the sum in the right hand side can be written as 
$$\frac{\{n+m\}!}{\{1\}\{n-m-1\}!}\{n+m+1\}\{n-m-1\}P(q,q^n,q^{2m+2}).$$ 
We have
$$\{n+m+1\}\{n-m-1\} =  q^{2n} + q^{-2n} - q^{2 (m+1)} - q^{-2(m+1)}$$
So the sum in the right-hand side of (\ref{eq:1})  may be expressed as 
$$\bigl\langle (M^2+M^{-2}-q^2 x-q^{-2} x^{-1})P(q, M , q^2x)  \bigr\rangle.$$
The two other equations are proved in the same way.
\end{proof}

As an application of these formulas, we prove two recursion formulas, the first one being already obtained using the computer by S. Garoufalidis and X. Sun in \cite{gs}.
\begin{prop} \label{sec:recurrence-relations-single}
The function $J_n^8$ satisfies the identity $QJ^8_n=R$ where
\begin{xalignat*}{2}
Q&=(q^{-1}M^2-qM^{-2})L+(qM^2-q^{-1}M^{-2})L^{-1}\\
&+(M^2-M^{-2})(-M^4-M^{-4}+M^2+M^{-2}+q^2+q^{-2}), \\
R& =D(M^5+M^{-5}+M^3+M^{-3}-(q^2+q^{-2})(M+M^{-1}))
\end{xalignat*}
and $D=(q-q^{-1})^{-1}$.
\end{prop}

\begin{proof} 
Apply the three formulas of Proposition \ref{rec1} to $P=1$. This gives:
\begin{xalignat}{3}\label{jones-decale}
\langle 1\rangle= &(M^2+M^{-2})\langle 1\rangle-q^2\langle x\rangle-q^{-2}\langle x^{-1}\rangle+D(M-M^{-1})\\
L\langle 1\rangle= & -(q+q^{-1})\langle 1\rangle+q^3M^2\langle x\rangle+q^{-3}M^{-2}\langle x^{-1}\rangle +D(qM-q^{-1}M^{-1})\notag\\
L^{-1}\langle 1\rangle= & -(q+q^{-1})\langle 1\rangle+q^3M^{-2}\langle x\rangle+q^{-3}M^2\langle x^{-1}\rangle +D(q^{-1}M-q M^{-1})\notag
\end{xalignat}
Computing the linear combination of $L\langle1\rangle,L^{-1}\langle1\rangle$ and $\langle1\rangle$ given by $Q\langle 1 \rangle$, all the brackets cancel and we obtain $R$.
\end{proof}

\begin{prop} \label{sec:recurrence-relations-system}
Let $J_\ell=\langle 1\rangle=J_\ell^8$ and $I_\ell=\langle x\rangle$. Then  $(J_\ell,I_\ell)$ satisfies the following $q$-difference equation:
\begin{eqnarray*}
\begin{pmatrix} J_{\ell+1}\\ I_{\ell+1} \end{pmatrix}&=&\begin{pmatrix}
q^{-1}M^{-4}-q^{-1}M^{-2}-q & q^3 M^2 -q M^{-2}\\
q^{-1}M^{-2}-qM^2& q^3M^4-qM^2-q
\end{pmatrix}\begin{pmatrix} J_\ell\\ I_\ell \end{pmatrix}\\
&&+D\begin{pmatrix}qM-q^{-1}M^{-3}\\ qM^3-q^{-1}M^{-1}\end{pmatrix}
\end{eqnarray*}
where $D=(q-q^{-1})^{-1}$.
\end{prop}
\begin{proof}
The first line of the system follows from the two first formulas of (\ref{jones-decale}). From the first one, we deduce that
$$\langle x^{-1}\rangle=q^2(M^2+M^{-2}-1)\langle 1\rangle -q^4\langle x\rangle +q^2D(M-M^{-1}).$$
The second one implies 
\begin{xalignat*}{2} 
L \langle 1 \rangle = & q^3 M^2 \langle x \rangle + q^{-3} M^{-2} \langle x^{-1} \rangle - ( q + q^{-1} ) \langle 1 \rangle + D( q M - q^{-1} M^{-1}) \\
= & q^3 M^2 \langle x \rangle + q^{-1} ( 1 + M^{-4} - M^{-2} ) \langle 1 \rangle - q M^{-2} \langle x \rangle + q^{-1} D (M ^{-1} -  M^{-3})\\  & -  ( q + q^{-1} ) \langle 1 \rangle + D( q M - q^{-1} M^{-1}) \\
\intertext{substituting the previous expression for $\langle x^{-1} \rangle$} 
= & (q^{-1} M^{-4} - q^{-1} M^{-2} - q ) \langle 1 \rangle + (q^3 M^2 - q M^{-2} ) \langle x \rangle\\ &  + D ( q M - q^{-1} M^{-3}) 
\end{xalignat*} 
This is the first line in the system of the proposition. The second line can be proved similarly by applying the two first formulas of Proposition \ref{rec1} to $P=x$. 
\end{proof}

The sequences $J_n$ and $I_n$ will be used for constructing knot states. It is then necessary that they belong to $\sP$ which is the purpose of the following lemma.
\begin{lem} \label{sec:exist-I}
For all $k>2$, one has $J_{n+2k}(-e^{i\pi/2k})=J_n(-e^{i\pi/2k})$ and the same formula holds for $I_n$.
\end{lem}
\begin{proof}
This equation is automatically satisfied for sequences of Jones polynomials (see \cite{LJ2}, Proposition 1.1 for a proof). Hence, we have to prove it only for $I_n$. From the fact that $\sP$ is a module over $\sT$ and that $J_n=\langle 1\rangle$ belong to $\sP$, we obtain that the three terms of equations \eqref{jones-decale} belong to $\sP$. As constant terms also belong to $\sP$ we obtain that the three expressions  
\begin{gather*}
q^2\langle x\rangle+q^{-2}\langle x^{-1}\rangle,\qquad q^3M^2\langle x\rangle+q^{-3}M^{-2}\langle x^{-1}\rangle, \\   q^3M^{-2}\langle x\rangle+q^{-3}M^2\langle x^{-1}\rangle
\end{gather*}
belong to $\sP$. Eliminating $\langle x^{-1}\rangle$, we find that $(1-q^2M^4)\langle x\rangle$ and $(1-q^2M^{-4})\langle x\rangle$ belong to $\sP$. This shows that $I_n$ satisfies the equation $$I_{n+2k}(-e^{i\pi/2k})=I_n(-e^{i\pi/2k})$$ unless one has $q^{2+4n}=1=q^{2-4n}$ for $q=e^{i\pi/k}$. This is impossible if $q^4\ne 1$ which is assumed by the hypothesis $k>2$. 
\end{proof}

\section{Microsupport and abelian representations}\label{sec:micro_abelian}

\subsection{Moduli spaces} \label{sec:moduli-spaces}

For any compact manifold $X$, we define the moduli space
$\mo (X)$ as the set of conjugacy classes of representations $\rho :
\pi_1 (X) \rightarrow \su$. In particular, given a knot $K$, we consider the moduli spaces of the exterior $E_K$ and of the peripheral torus $\Si$. Since $\Si$ is the boundary of $E_K$ there is a natural restriction map 
$$  r : \mo ( E_K) \rightarrow \mo ( \Si)$$
As in the previous sections let $E = H_1(\Si, \R)$ and $R= H_1 ( \Si, \Z)$. There is a natural projection  $ \pi : E \rightarrow \mo ( \Si)$ defined by: 
\begin{gather} \label{eq:coordonnees-mo-Sigma}
 \pi (x) (\gamma) = \exp \bigl( (\gamma \cdot x ) D \bigr) , \qquad x \in E, \; \gamma \in H_1( \Si, \Z)
\end{gather}
where the $\cdot$ stands for the intersection product and  $D$ is the diagonal matrix with coefficient $2i\pi$, $-2i \pi$. The map $\pi$ is onto and induces a bijection between $\mo ( \Si)$ and the quotient of $E$ by the group $R \rtimes \Z_2$. 

$\mo (E_K)$ is the disjoint union of  $\mo^{\ab}$ and $\mo ^{\ir}$, which consist respectively of the abelian and irreducible representations. Since $H_{1}( E_K, \Z) $ is infinitely cyclic,  $\mo^{\ab}$ is homeomorphic to an interval. Let  $\Delta_K\in\Z[t^{\pm 1}]$ be the Alexander polynomial of $K$ normalized in such a way that $$\Delta_K(1)=1 \quad  \text{ and } \quad \Delta_K(t^{-1})=\Delta_K(t). $$  We say that a representation $\rho \in \mo ^{\ab}$ is regular if $\Delta_K( u ) \neq 0$ for any eigenvalue $u$ of $\rho (\mu)^2$.

Since $H_{1}( E_K, \Z) $ is generated by $\mu$, we can parameterize $\mo^{\ab}$ by $[0,1]$ by sending $t \in [0,1]$ into the representation $\rho\in \mo^{\ab}$ such that $$ \rho ( \mu ) = \exp ( t D/2)$$ with $D$ the same matrix as above. The restriction of $r$ to $\mo^{\ab}$ is injective. It sends the the representation $\rho$ parameterized by $t$ into $\pi ( t \la /2)$. We say that a point $x \in r ( \mo ^{\ab})$ is regular if $x$ is the restriction of a regular abelian representation and  a neighborhood of $x$ in $\mo (\Si)$ does not meet $r( \mo ^{\ir})$.  

\subsection{Examples} 

Let us consider first the torus knot with parameter $(a,b)$. Its Alexander polynomial is 
$$ \Delta_{a,b} ( t) = t^{\frac{1}{2} ( a + b  -ab -1)} \frac{ (t-1) (t^{ab} -1)}{( t^a -1 ) ( t^b -1 )} .$$
It has $2N$ roots with $N= (a-1)(b-1)/2$. $\mo (E_K)$ consists of $N$ open intervals of irreducible representations attached  to $\mo ^{\ab}$ as follows
$$ \mo (E_K) = \Bigr( \mo^{\ab} \amalg  \coprod_{i=1,\ldots, N}  I_i \Bigl)/\Phi  $$ 
Here $I_i$ is a closed interval, $\mo^{\ab} \simeq [0,1]$ as in the previous section and $\Phi$ is a bijection from  $\coprod_{ i = 1 , \ldots, N} \partial I_i$ to 
$$ X= \{ t \in [0,1]/  \Delta_{a,b}( e^{2i\pi t}) =0 \} \subset \mo^{\ab}.$$ 
For each $i$, the restriction $r$ sends each $I_i$ into $\pi( (ab \mu - \la ) \R)$ or $\pi ( (ab \mu - \la ) \R + \frac{1}{2ab} \la )$. But $r$ is not injective. Any point of $r \bigl( \mo ^{\ab} \setminus \{ \ell / ab  ;\; \ell =0, \ldots, N\} \bigr)$ is regular. The points $r(\ell / ab)$ with $\ell$ a multiple of $a$ or $b$ may or may not be regular. For instance, for the trefoil knot, $r(0)$ is regular whereas $r(1/2)$ is not.   

For the figure eight knot, the Alexander polynomial is 
$$\Delta_8( t)  = 3 - t - t^{-1}$$
It has no zero in the unit circle. The set $\mo^{\ir}$ of irreducible representations is a circle. The restriction of $r$ to $\mo^{\ir}$ is an immersion with one double point. Its image is $\pi (X_8)$ with 
\begin{xalignat}{2} \label{eq:X_8}
  X_8 = & \bigl\{  p \mu + q \la / \; \cos ( 2 \pi p ) + 1 = \cos ( 8 \pi q ) - \cos ( 4 \pi q) \bigr\} \\ & \quad \setminus \bigl( \mu/2 + \mu \Z  + \la/2 \Z \bigr)  \notag
\end{xalignat}
Furthermore $r( \mo^{\ab} )\cap r( \mo^{\ir}) = \{ \pi( \frac{1}{4} \la) \}$. So all the points of $r( \mo^{\ab})$ are regular except $ \pi( \frac{1}{4} \la)$.

\subsection{Microsupport}

Our first conjecture says that the knot state $(Z_k)$ concentrates on the image $r(\mo (E_K))$ as the level $k$ becomes large. Let us first recall the definition of microsupport.  Consider a family $( \xi_k \in \Hilb_k; \; k \in \Z_{>0})$. We say $(\xi_k)$ is {\em admissible} if there exists a positive $C$ and an integer $M$ such that 
$$\bigl\|  \xi_k  \bigr\|_{\Hilb_k} \leqslant C k^{M} .$$ 
We say that $(\xi_k)$ is a $O(k^{-\infty})$ on an open set $U$ of $E$ if for any $x \in U$, there exists a a neighborhood $V$ of $x$ and a sequence $(C_M)$ of positive numbers such that for any $M$ and any $k$, 
$$  | \xi_k (y) | \leqslant C_M k^{-M}, \qquad \forall y \in V $$ 
The {\em microsupport} of $(\xi_k)$ is the smallest closed subset $F$ of $E$ such that $(\xi_k)$ is a $O(k^{-\infty})$ on $E \setminus F$. We denote it by $\MS( \xi_k)$. The microsupport is a closed $R$-invariant subset of $E$. If the family $(\xi_k)$ consists of alternating sections, its microsupport is also invariant by $-\id_E$. 

Consider the knot state $Z_k(E_K) \in V_k( \Si, \nu)$  as in Section \ref{sec:knot-state} and identify $V_k(\Si,\nu)$ with the subspace of alternating sections of $\Hilb_k( j, \delta)$ as in Section \ref{sec:equiv}.  

\begin{conj} \label{sec:conj_microsupport} 
The microsupport of the knot state $(Z_k(E_K))_k$ satisfies
$$ \pi  \bigl( \MS ( Z_k(E_K )) \bigr) \subset  r ( \mo (E_K)).$$
\end{conj}

A stronger version of the conjecture is the equality between $r ( \mo (E_K))$ and  $\pi  \bigl( \MS ( Z_k(E_K )) \bigr)$. We do not know any counter-example.

If $K$ is the trivial knot, $\mo( E_K)$ consists only of abelian representation, and it follows from Theorem \ref{theo:solid-torus-state} that $\pi  (\MS( Z_k(E_K)))$ is exactly  $r ( \mo (E_K)) $. 

\subsection{Abelian representations}  \label{sec:conj_abel_repr}

Our second conjecture describes the knot state on a neighborhood of any regular abelian representation. Denote by $\mu, \la \in R$ the classes of the meridian and longitude respectively. Recall that $\mo ^{\ab}$ maps onto $\pi (  \R \la)$. We say that a point $x \in \la \R$ is regular if $\pi( x)$ is regular as defined in Section \ref{sec:moduli-spaces}. It means that a neighborhood of $\pi(x)$ does not meet $r( \mo ^{\ir})$ and $\Delta_K ( e^{4i \pi q }) \neq 0$ if $ x = q \la$. 

\begin{conj}\label{conj:abelian}
Any regular point of $ \R \la$ has an open neighborhood $V$ in $E$ such that $V \cap \R \la$ consists of
regular points and 
$$ Z_k( E_K) (x)   =   e^{im \frac{\pi}{4}} \Bigl( \frac{k}{2 \pi} \Bigr)^{1/4}  t_\la^k(x) \otimes f(x ,k ) \Om_{\la} + O(k^{-\infty}) , \qquad \forall x \in V $$
where $m$ is an integer,  
\begin{itemize} 
\item[-] $t_\la$ is the holomorphic section of $L$ restricting to $1$ on $\R \la$.
\item[-] $f(\cdot , k)$ is a sequence of $\Ci (V)$ admitting an
  asymptotic expansion of the form $f_0 + k^{-1} f_1 +  $ for the
  $\Ci$ topology with holomorphic coefficients $f_i$ defined on $V $ and  $$f_0(q\la) =  \frac{1}{\sqrt{2}}  \frac{ \si - \si^{-1}  }{\Delta_K (\si^2)}  \quad \text{ with } \quad \si = e ^{ 2i \pi q }  $$ for any $q \la \in V$.
\item[-] $\Om_{\la}\in \delta$ is such that $\Om_\la ^2 ( \la) = 1$.   
\end{itemize}
\end{conj}

The section $t_\la $ can be explicitly computed in terms of the coordinates $p, q$ dual to the base $(\mu, \la)$. Let  $\tau$ be the complex parameter defined as in Section \ref{sec:torus-quantization}, then 
$$ t_\la = \exp \bigl( - 2 i \pi \bigl( q + \frac{p}{\tau} \bigr) p\bigr).  $$
For the trivial knot, the Alexander polynomial being the constant 1, the conjecture follows from Theorem \ref{theo:solid-torus-state}.

\subsection{Admissibility}

In this section we show that for any knot $K$ the family $(Z_k(E_K), \; k \in \Z_{>0} )$ is admissible, that is there exists a positive $C$ and an integer $M$ such that $||Z_k(E_K)||\le Ck^M$ for any $k$. Using the isomorphism of Theorem \ref{sec:equiv-top-geom}, this norm may be computed in the framework of TQFT. Using the axioms, one has 
$$\langle Z_k(E_K),Z_k(E_K)\rangle=Z_k(E_K\cup(-E_K)).$$ The manifold $E_K\cup(-E_K)$ is a closed 3-manifold called the double of the knot. The admissibility is then a consequence of the following lemma.

\begin{lem}
For any closed 3-manifold $M$, there exists a positive constant $C$ and an integer $N$ such that $|Z_k(M)|\le Ck^N$.
\end{lem}

This was proved by Garoufalidis in \cite{gar1} but we reproduce the argument for completeness. 

\begin{proof}
Recall that any closed 3-manifold has a Heegard decomposition, that is, it is obtained by gluing two handlebodies along their boundaries. 
Let $H_g$ be a handlebody of genus $g$ and $\Sigma_g$ be its boundary. A Heegard decomposition of $M$ of genus $g$ is an homeomorphism of $M$ with $H_g\cup_{\phi}(-H_g)$ for some diffeomorphism $\phi$ of $\Sigma_g$. Using the axioms of TQFT, this decomposition gives 
$$Z_k(M)=\langle Z_k(H_g),Z_k(\phi)Z_k(H_g)\rangle .$$ 
Using the unitarity of $Z_k(\phi)$ and Cauchy-Schwarz inequality, we get 
$$|Z_k(M)|\le ||Z_k(H_g)||^2=Z_k(H_g\cup(-H_g)).$$ 
On the other hand, the manifold $H_g\cup(-H_g)$ is homeomorphic to a connected sum of $g$ copies of $S^2\times S^1$. This gives $$Z_k(H_g\cup(-H_g))=\frac{Z_k(S^2\times S^1)^{g}}{Z_k(S^3)^{g-1}}.$$ 
As $Z_k(S^2\times S^1)=k-1$ and $Z_k(S^3)=\sqrt{2/k}\sin(\pi/k)$ we get 
$$||Z_k(H_g)||^2\underset{k\to\infty}{\sim} \frac{k^{5g/2-3/2}}{(\sqrt{2}\pi)^{g-1}}$$ which ends the proof.
\end{proof}

\subsection{Some microlocal results} \label{sec:basic-micr-prop}

Let us recall some basic results on microsupport and Toeplitz operators. 
First Toeplitz operators reduce microsupport in the following sense. For any admissible family $(\xi_k)$ and Toeplitz operator $(T_k)$, we have 
\begin{gather} \label{eq:top_micsup}
\MS( T_k \xi_k) \subset \MS ( \xi_k).
\end{gather}
Let $(T_k)$ and $(S_k)$ be two Toeplitz operators of $\Hilb_k$. Assume that $T_k \equiv S_k$ on an open set $U$ of $E$, in the sense that the total symbols of $T_k$ and $S_k$ are the same on $U$. Then for any admissible family  $(\xi_k)$, we have $ T_k \xi_k = S_k \xi_k + O(k^{-\infty})$ on $U$.

Our proof of Conjectures \ref{sec:conj_microsupport} and \ref{conj:abelian} for the torus knots and the figure eight knot starts with the analysis of the state 
\begin{gather} \label{eq:def_ground_state}
 Z^0_k = \frac{1}{2i \sqrt k} \sum_{\ell \in \Z / 2k \Z}   \Psi_\ell 
\end{gather}
of $\Hilb_k$.  Introduce as in the statement of Conjecture \ref{conj:abelian} the section $t_\la$ of $L$ and the vector $\Om_\la$ of $\delta$.  
\begin{theo} \label{sec:ground_state}
Let  $(T_k)$ be a Toeplitz operator of $\Hilb_k$ with principal symbol $f \in \Ci( E)$.
Then the microsupport of $(T_k Z^0_k)$ is contained in $\R \la + \Z \mu$. On a neighborhood of $\R \la$ one has 
\begin{gather} \label{eq:toep_lag}
 T_k Z_k^0 =  \frac{ e^{3i \pi/4}} { \sqrt 2} \Bigl( \frac{k}{2\pi} \Bigr)^{1/4} g( \cdot, k) t_\la^k  \otimes \Om_{\la}  + O( k^{-\infty} ) 
\end{gather}
where $g(\cdot, k)$ is a sequence of $\Ci (E)$ admitting an asymptotic expansion of the form $ g_0 + k^{-1} g_1 + \ldots $ on a neighborhood of $\R \la$. Furthermore $g_0 = f $ on $\R \la$. 
\end{theo}

 Observe also that the points $x \in \R \la$ such that $g_0 (x) \neq 0$ belong to the microsupport of $T_k Z^0_k$.  

\begin{proof} 
In the proof of Proposition \ref{sec:geo_isomorphism}, we showed that the value of $\frac{1}{\sqrt{2k}}\sum   \Psi_\ell $ at $0$ is 
$$   e^{i \pi/4}\Bigl( \frac{k}{2\pi} \Bigr)^{1/4}  \Theta ( 0, -1/\tau) \Om_{\la}  =   e^{i \pi/4}\Bigl( \frac{k}{2\pi} \Bigr)^{1/4}  \Om_{\la} +   O( k^{-\infty} )
$$
Furthermore $\sum   \Psi_\ell $ is an eigenstate of $T_{\la/2k} ^* $ with eigenvalue $1$. So if we let $( \Psi'_\ell)$ be one of the two basis of $\Hilb_k$ associated to the basis   $(\la, - \mu)$ of $R$ as in Proposition \ref{sec:rep_Heisenberg}, we have that 
$$ Z^0_k = \pm \frac{ e^{3i \pi/4}} { \sqrt 2} \Psi'_0 +  O( k^{-\infty} )$$
Now by Proposition \ref{sec:base_lagrangian_states}, the microsupport of $(Z_k^0)$ is $\R \la + \Z \mu$ and we have 
$$  Z_k^0  = \frac{ e^{3i \pi/4}} { \sqrt 2} \Bigl( \frac{k}{2\pi} \Bigr)^{1/4} t_\la^k  \otimes \Om_{\la}  + O( k^{-\infty} ) $$ 
on a neighborhood of $\R \la$. Since Toeplitz operators reduce microsupport, the  microsupport of $(T_k Z^0_k)$ is contained in $\R \la + \Z \mu$. The second part of the result follows from  Proposition 2.7 of \cite{l3} and Theorem 3.3 of \cite{l4} describing the action of a Toeplitz operator on a Lagrangian state.
\end{proof}

\subsection{Torus knots and figure eight knot} 

In this part we prove Conjecture \ref{conj:abelian} and compute partially the knot state microsupport for the torus knots and the figure eight knot. During the proof we use the operators 
$$ M = T^* _{ \mu/2k} , \qquad L =  T^*_{- \la /2k} $$
By Theorem \ref{sec:Toeplitz_operators}, these are Toeplitz operators whose symbols are respectively 
$$ \sM = \exp ( -2 i \pi q) , \qquad \sL = \exp ( - 2 i \pi p) $$
where $p$ and $q$ are the linear coordinates of $E$ dual to  the basis $( \mu, \la)$. Observe that the restriction of $\sM$ and $\sL$ at $\la \R$ are respectively $\si^{-1} $ and 1 where $\si$ is the function on $\R \la$ given by $\si ( q \la) = e^{2i \pi q}$.

Let us give a second proof of Conjectures \ref{sec:conj_microsupport} and \ref{sec:conj_abel_repr} for the trivial knot. Recall that its $\ell$-th Jones polynomial is
$$  J_\ell (t) = \frac{ t^{2 \ell} - t^{-2 \ell} } { t^2  -t ^{-2}}$$   
Let $D_0 = \la \R + \mu \Z$.

\begin{prop} 
The microsupport of the trivial knot state is $D_0$. Furthermore conjecture \ref{conj:abelian} holds for the trivial knot. 
\end{prop}

\begin{proof} 
Using the above formula for the Jones polynomial and the expression (\ref{eq:etat_noeud}) for the knot state, we deduce that the trivial knot state is  $$    (M - M^{-1} )Z_k^0 $$
where $Z^0_k$ is given by (\ref{eq:def_ground_state}). 
So by Theorem \ref{sec:ground_state}, the microsupport of $Z_k$ is contained in $D_0$. 
The restriction of the symbol of $M - M^{-1}$ to $\R \la$ is $-\si + \si ^{-1}$. So Conjecture \ref{conj:abelian} follows from the second part of  Theorem \ref{sec:ground_state}. Furthermore since $-\si + \si ^{-1}$ vanishes at isolated points, the microsupport of $Z_k$ is $D_0$.
\end{proof}

Let us consider the torus knots. Denote by $E_{a,b}$ the exterior of the torus knot with parameter $(a,b)$. For any rational $s$, denote by $D_s$ the subset $(s \mu - \la) \R + R $ of $E$.

\begin{theo} \label{prop:torus_ab_preuve}
The microsupport of $(Z_k( E_{a,b}))$ satisfies $$D_0 \subset \MS( Z_k ( E_{a,b}) ) \subset D_0 \cup D_{ab} \cup  ( D_{ab} + \tfrac{1}{2ab} \la).$$
Conjecture \ref{conj:abelian} holds for torus knot with parameter $(a,b)$ for any point in $\R \la \setminus \frac{1}{2ab} \N \la$.  
\end{theo}

Conjecture \ref{sec:conj_microsupport} holds also for the torus knots but the proof requires more analysis, cf \cite{Ctoric}.

\begin{proof}
The recurrence equation of Proposition \ref{sec:recurrence-relations-torique} implies that 
\begin{gather} \label{eq:eq_torus}
   \bigl( \id - \tilde  M^{-2 ab} L^{-2}  \bigr)  Z_k (E_{a,b}) = T Z^0_k
 \end{gather}
 with $\tilde M = q^{-1} M$ and
\begin{xalignat*}{2}
 & T = \tilde M ^{-ab} \bigl(  q  \tilde  M ^{a+b}-  q^{-1} \tilde M ^{a-b} -   q^{-1} \tilde M ^{-a +b}   + q  \tilde M ^{- a - b }\bigr).
\end{xalignat*}
The symbol of $\id - \tilde  M^{-2 ab} L^{-2}$ vanishes exactly on $D_{ab} \cup( \frac{1}{2ab} \la +  D_{ab} )$. By the symbolic calculus of Toeplitz operator, for any $x \notin D_{ab} \cup( \frac{1}{2ab} \la +  D_{ab} )$ there exists a Toeplitz operator $S$ such that $S(\id - \tilde  M^{-2 ab} L^{-2}) \equiv \id$ on a neighborhood of $x$. So Equation \eqref{eq:eq_torus} implies that
$$  Z_k (E_{a,b}) = S T Z^0_k + O( k^{-\infty})$$ on a neighborhood of $x$. If $x \notin D_0$ we conclude from Theorem \ref{sec:ground_state} that $ Z_k (E_{a,b})$ is a $O(k^{-\infty})$ on a neighborhood of $x$. If $x \in D_0$, we have that $ Z_k (E_{a,b})$ has the form (\ref{eq:toep_lag}) on a neighborhood of $x$ with $ g_0$ equal on  $\R \la $ to the symbol of $ST$. The symbol of $S$ is equal to the inverse of the symbol of $\id - \tilde  M^{-2 ab} L^{-2}$ on a neighborhood of $x$.  We deduce that on $\R \la$ 
\begin{xalignat*}{2}
 g_0 &  = \frac{ \si^{ab} ( \si^{a+b} - \si^{a-b } - \si^{ -a + b} + \si^{-a -b } )}{1 - \si^{2ab} } \\
 & = \frac{ ( \si^a - \si ^{-a} ) ( \si^ b  - \si^{ -b} ) }{\si^{-ab} - \si ^{ab}} \\
 & = -  \frac{ \si - \si^{-1} } { \Delta_{a,b} ( \si^2)}
\end{xalignat*} 
which ends the analysis of Proposition \ref{sec:recurrence-relations-torique}.
\end{proof}

For $a=2$ we can deduce a more precise result from proposition \ref{sec:recurrence-relations-torique2}. 
\begin{theo} 
The microsupport of $(Z_k( E_{2,b}))$ satisfies $$D_0 \subset \MS( Z_k ( E_{2,b}) ) \subset D_0  \cup  ( D_{2b} + \tfrac{1}{4b} \la).$$ 
Conjecture \ref{conj:abelian} holds for torus knot with parameter $(2,b)$ for any point in $\R \la \setminus (\frac{1}{4b} + \frac{1}{2b} \N ) \la$.  
\end{theo}

\begin{proof}
The recurrence formula of Proposition \ref{sec:recurrence-relations-torique2} implies that the following equation holds:
\begin{gather} \label{eq:eq_torus_2}
   \bigl( \id + q^b M^{-2 b} L^{-1}  \bigr)  Z_k (E_{2,b}) = q^bM^{-b}(q^{-1}M^2-qM^{-2}) Z^0_k
 \end{gather}
 We deduce from it the assertion on microsupport because the symbol of the operator in the left hand side of Equation 
\eqref{eq:eq_torus_2} vanishes only on $D_{2b} + \tfrac{1}{4b} \la$. Furthermore we obtain the same expression for $Z_k(E_{2,b})$ on $D_0$ as the one obtained for general torus knots, which proves the second assertion. 
\end{proof}

Let us consider now the figure eight knot. Denote by $E_8$ its exterior.
\begin{theo} \label{sec:8_theo}
The microsupport of $Z_k(E_8)$ satisfies $$D_0 \subset \MS( Z_k(E_8) ) \subset D_0 \cup X_8$$
where $X_8$ is defined in (\ref{eq:X_8}). 
Conjecture \ref{conj:abelian} holds for the figure eight knot state.
\end{theo}

We will prove in \cite{LJ2} that the microsupport is exactly $D_0 \cup X_8$. 

\begin{proof} 
The proof is essentially the same as for the torus knots except that we use the system of equation given in Proposition \ref{sec:recurrence-relations-system}. Actually the single equation $QJ=R$ of Proposition \ref{sec:recurrence-relations-single} is not sufficient to conclude. Recall that we introduced in Proposition \ref{sec:recurrence-relations-system} a sequence $I \in \sP$. Let $Y_k = Z_k (I)  \in \Hilb_k$. By Proposition \ref{sec:recurrence-relations-system}, we have the system
$$ \begin{cases}  A Z_k(E_8) + B Y_k =  S Z^0_k \\  C Z_k(E_8)  + D Y_k =  T Z^0_k
\end{cases} $$
where $A$, $B$, $C$, $D$, $S$ and $T$ are Toeplitz operators whose symbols are respectively given by 
\begin{gather*} 
 a = \sL - \sM^{-4} + \sM ^{-2} +1, \quad b = -c = -\sM^2 + \sM^{-2}    \\
d = \sL - \sM^4 + \sM^2 + 1, \quad  s = \sM - \sM^{-3}, \quad t = \sM^3 - \sM^{-1}  
\end{gather*}
A straightforward computation gives 
\begin{xalignat*}{2}
 ad -bc = &  \sL \bigl( \sL + \sL^{-1} +2 - ( \sM^4 + \sM^{-4} ) + \sM^2 + \sM^{-2} \bigr) \\
= & 2 \sL \bigl( \cos ( 2 \pi p)  + 1 - \cos ( 8 \pi q ) + \cos ( 4 \pi q) \bigr)
\end{xalignat*}
Let $x \in E$ not belonging to  $X_8$. Assume $x \notin  \mu/2 + \mu \Z + \la/2 \Z $. Then the determinant $ad -bc$ does not vanish at $x$. It follows from the symbolic calculus of Toeplitz operators that there exists Toeplitz operators $E$, $F$, $G$, $H$ such that
$$ \begin{pmatrix} E &F\\ G&H \end{pmatrix}  \begin{pmatrix} A&B\\ C&D\end{pmatrix} = \begin{pmatrix}  \id &0\\ 0 & \id \end{pmatrix}
$$
on a neighborhood of $x$. Hence
$$ Z_k(E_8)  = ( ES + FT) Z^0_k + O( k^{-\infty}) 
$$ 
on a neighborhood of $x$. From now on we argue as in the proof of Theorem \ref{prop:torus_ab_preuve}. If $x \notin D_0$, we conclude that $x$ does not belong to the microsupport of $Z_k(E_8)$. If $x \in D_0$, we conclude that $Z_k(E_8)$ is of the form (\ref{eq:toep_lag}) on a neighborhood of $x$ where $g_0$ is equal to the symbol of $ES + FT$ on $\R \la$. On a neighborhood of $x$, the symbols of $E$ and $F$ are respectively $ d ( ad -bc)^{-1}$ and $-b ( ad - bc )^{-1}$; so the symbol of $ES + FT$ is 
$$ \frac{ ds - bt}{ad - bc} = \frac{( \sM^{2} -\sM^{-2} )( \sM + \sL \sM^{-1})}{ ad - bc}   
$$ 
Using the previous formula for $ad -bc$, we obtain that the restriction of $g_0$  to $\R \la$ is given by
$$  - \frac{ (\si^2 - \si^{-2})( \si + \si^{-1})}{ 4 + \si^2 + \si^{-2} - ( \si^4 + \si^{-4})} = - \frac{ \si - \si^{-1} } { 3 - ( \si^2 + \si^{-2})}   
$$ 
Conjecture \ref{conj:abelian} follows. 

To end the proof, we have to show that the points of $\mu/2 + \mu \Z + \la/2 \Z $ do not belong to the microsupport. To do that, we will use some microlocal results recalled in \cite{LJ2}. Consider the equation $QJ=R$ of Proposition \ref{sec:recurrence-relations-single}. The symbol of $Q$ is 
$$ 2 ( \sM^2 - \sM^{-2} )  \bigl( \cos ( 2 \pi p)  + 1 - \cos ( 8 \pi q ) + \cos ( 4 \pi q) \bigr)$$
It vanishes on $\mu/2 + \mu \Z + \la/2 \Z $, but its differential does not. Furthermore the microsupport of $R$ does not intersect $\mu/2 + \mu \Z + \la/2 \Z $. So by singularity propagation, we deduce from $QJ=R$ that the points of  $\mu/2 + \mu \Z + \la/2 \Z $ can not be isolated in the microsupport of $R$. Hence by the first part of the proof, they do not belong to the microsupport. The result on the propagation of singularities is an easy consequence of Theorem 5.2 of \cite{LJ2}. 
\end{proof}

\subsection{Relation with the AJ-Conjecture} \label{sec:relation-with-aj}

AJ conjectures were proposed by Garoufalidis in \cite{gar}.  We will state two versions, that we will call geometric and polynomial. The geometric version, Conjecture \ref{conj:AJ_conjecture_geom},  is easy to compare with Conjecture \ref{sec:conj_microsupport} on the microsupport of the knot state. The polynomial version, Conjecture \ref{conj:AJ_conjecture_polynomial}, has been proved by  Garoufalidis in \cite{gar} for the trefoil and the figure eight knot using the computer. Hikami showed it in \cite{hikami} for the torus knots and Le proved it for all 2-bridge knots with irreducible $A$-polynomial, see \cite{le}. It gives an important information on the microsupport of the knot state which is in general not sufficient to deduce Conjecture \ref{sec:conj_microsupport}. 

Following \cite{ccgls}, the deformation variety of a knot $K$ in $S ^3$ is the subvariety of $(\C^*)^2$ defined as follows. For any compact manifold $M$, let $X(M)$ be the algebro-geometric quotient of the affine variety $$R(M) = \operatorname{Hom} ( \pi_1(M), \operatorname{Sl}_2 (\C))$$ under the conjugation action. If $\Si$ is the peripheral torus of the knot, we let $\Delta \subset R (\Sigma)$ be the variety of diagonal representations. $\Delta$ is isomorphic to $(\C^*)^2$, trough the map sending $(u,v)$ to $\rho \in \Delta$ such that $\rho ( \mu) = \operatorname{diag}(u,u^{-1})$ and  $\rho ( \la) = \operatorname{diag}(v,v^{-1})$.
The restriction $\bpi: \Delta \rightarrow X (\Sigma )$ of the canonical projection is a surjection which is generically two to one. The deformation variety of $K$ is 
$$ D_K =  \bpi^{-1} ( r ( X (E_K)))$$ 
where $r : X (E_K) \rightarrow X( \Sigma )$ is the restriction map.  
 
Let $J^K\in \sP$ be the sequence of the colored Jones polynomials of $K$. We define the ideal of recurrence relations for $J^K$ by $$\sA^K=\{P\in \sT, P J^K=0\}.$$
Let $\ep:\sT\to \C[M^{\pm 1},L^{\pm 1}]$ be the algebra homomorphism sending $q$ to $1$ and fixing $M$ and $L$. Garoufalidis defines in \cite{gar} the characteristic variety of $\sA^K$ in the following way:
\begin{equation} \label{eq:caracteristic_variete}
\operatorname{ch}(\sA^K)=\{(x,y)\in (\C^*)^2, P(x,y)=0\quad\forall P\in \ep(\sA^K)\}
\end{equation}
Given two subsets $A,B\in (\C^*)^2$, we say that they are $M$-essentially equal if there is a finite set $Y\in \C^*$ such that $A\cup \left(Y\times\C^*\right)=B\cup \left(Y\times\C^*\right)$. The following conjecture was proposed by Garoufalidis in \cite{gar}:
\begin{conj}[AJ-conjecture, geometric version] \label{conj:AJ_conjecture_geom}
The deformation variety $D_K$ and the characteristic variety $\operatorname{ch}(\sA^K)$ of $(\C^*)^2$ are $M$-essentially equal.
\end{conj}
 
Let us compare with Conjecture \ref{sec:conj_microsupport}. Recall that the moduli space $\mo ( \Sigma )$ is isomorphic with the quotient of $E$ by $R \rtimes \Z_2$. Let $\pi_\R$ be the projection from $E / R$ to $\mo ( \Sigma )$. The real analog of the deformation variety $D_K$  is $\pi_\R^{-1} ( r ( \mo (E_K))$. Furthermore, if we embed $E/R$ into $\Delta$ with the map sending $[p \mu + q \la]$ into $( e^{2i\pi p} ,e^{2i\pi q})$, we have that
\begin{gather} \label{eq:reel_comp}
  \pi_\R^{-1} ( r ( \mo (E_K)) \subset (E /R) \cap D_K . 
\end{gather}    
In general, the inverse inclusion is false, as $(E /R) \cap D_K $ is Zariski closed and $\pi_\R^{-1} ( r ( \mo (E_K)) $  is not necessarily closed, cf. for instance the trefoil knot.

Similarly the microsupport may be viewed as a real analog of the characteristic variety of $\operatorname{ch}(\sA^K)$ because of the following characterization. Consider the set ${\mathcal{I}}_K$ of Toeplitz operators $(T_k)$ of ${\mathcal{H}}_k$ such that 
$$ T_k Z_k (E_K) =  O(k^{-\infty})  $$
The principal symbol of a Toeplitz operator $(T_k)$ is an $R$-invariant function of $E$, so it descends to a function on $E/R$ . We denote it by $\si (T_k)$. Recall that the microsupport of $Z_k (E_K)$ is $R$-invariant, so it may be viewed as a subset of $E/R$.  
\begin{prop} \label{prop:micro_supp_Toeplitz}
The microsupport of $Z_k (E_K)$ is given by 
\begin{gather} \label{eq:carac_microsupport}
 \MS (Z_k ( E_K)) = \{ x \in E/R ; \; \si ( T_k) (x) = 0, \forall \; (T_k) \in {\mathcal{I}}_K \} .
\end{gather}
\end{prop}

\begin{proof} 
This follows from the properties of Toeplitz operators recalled in the beginning of Section \ref{sec:basic-micr-prop}. 
Assume that $x \in E$ does not belong to the microsupport of $Z_k ( E_K)$ so that on a neighborhood $V$ of $x$, $Z_k (E_K)$ is a $O(k^{-\infty})$. Then for any function $f \in \Ci _R (E)$ whose support is contained in $ V + R$, the Toeplitz operator $(\Pi_k M_f)$ belongs to the ideal ${\mathcal{I}}_k$ and its principal symbol is $f$.

Conversely assume that there exists a Toeplitz operator $(T_k) \in {\mathcal{I}}_K$ whose symbol does not vanish at $x$. Multiplying $(T_k)$ by a microlocal inverse, we may assume that $T_k \equiv \id_{{\mathcal{H}}_k}$ on a neighborhood of $x$. Consequently, we have on a neighborhood of $x$ that 
 $$Z_k(E_k) = T_k Z_k (E_K) + O(k^{-\infty}) = O(k^{-\infty})$$ 
because $(T_k)$ belongs to ${\mathcal{I}}_K$.
\end{proof}

Not only is the characterization of the microsupport (\ref{eq:carac_microsupport}) similar to the definition of the characteristic variety (\ref{eq:caracteristic_variete}), but we also have the following relation 
\begin{gather} \label{eq:microsupport_caracteristique}
 \MS ( Z_k ( E_K)) \subset (E/ R) \cap \operatorname{ch}(\sA^K)
\end{gather}
similar to (\ref{eq:reel_comp}). This follows from the fact each recurrence relation $P$ in ${\mathcal{A}}^K$  gives a Toeplitz operator $(T_k)$ in ${\mathcal{I}}_K$ by Theorem \ref{sec:Toeplitz_operators}. Furthermore the restriction of $\ep (P)$ to $E/R$ is the principal symbol of $(T_k)$.  As for (\ref{eq:reel_comp}), the inclusion of (\ref{eq:microsupport_caracteristique}) is strict. The reason is the same, the microsupport is not necessarily Zariski closed.  

So Conjecture \ref{sec:conj_microsupport} is a real analog of the AJ-conjecture, where the group $\operatorname{Sl}_2(\C)$ is replaced by $\su$. An important difference is that the AJ conjecture deals with algebraic set whereas our conjecture compare two closed sets for the usual topology.  Another difference is that the equality in the AJ-conjecture is an equality up to some lines. The reason is that in the polynomial version (the only version which is proved for some knots) we compare two polynomials up to some $M$ factor.  Actually we do not know any counter example of the equality between the deformation variety and the characteristic variety. 

To end this section, let us explain what we can deduce on the microsupport of the knot state from the polynomial version of the AJ conjecture. Let $A_K \in\C[M,L]$ be the $A$-polynomial introduced in \cite{ccgls}. For torus knots with parameters $a,b$ the $A$ polynomial is equal to $(L-1)(LM^{ab}+1)$
if $a=2$ and otherwise to 
$ (L-1)(L^2 M^{2ab}-1)$. 
The $A$ polynomial of the figure eight knot is $$(L-1)(L^2M^4+M^4+L(-M^8+M^6+2M^4+M^2-1)).$$
\begin{conj}[AJ-conjecture, polynomial version] \label{conj:AJ_conjecture_polynomial}
There exists a \\q-difference relation $\al \in {\mathcal{A}}^K$ such that
\begin{gather} \label{eq:AJ_conjecture_polynomial_version} 
  \ep (\al) = F A_K 
\end{gather}
for some fraction $F\in \C (M)$. 
\end{conj}
Moreover, it is shown in \cite{le} (Proposition 4.1) that for any $\alpha\in\mathcal{A}^K$, $\ep(\alpha)$ is divisible by $A_K$ in $\C(M)[L^{\pm 1}]$. Then if the $L$-degrees of $\alpha$ and $A_K$ coincide, the recurrence polynomial $\alpha$ satisfying Conjecture \ref{conj:AJ_conjecture_polynomial} is a generator of the recurrence ideal in a weak sense, see \cite{le}.

Using as above that $q$-difference relations gives Toeplitz operators in the ideal ${\mathcal{I}}_K$,  Equation (\ref{eq:AJ_conjecture_polynomial_version}) implies that 
\begin{gather} \label{eq:micro_AJ_conjecture}
 \MS ( Z_k (E_K)) \subset \bigl\{ A_K( e^{ -2 i \pi q} , e^{ -2 i \pi p}) = 0 \bigr\} \cup \bigcup_{i=1,\ldots , N} \{ q = q_i\} 
\end{gather}
where the zeros and poles of $F$ in the unit circle are $e^{2i\pi q_1}, \ldots, e^{2 i \pi q_N}$. One can check that for the $q$-difference relations proved in \cite{gar} and \cite{hikami} for the figure eight knot and the torus knots, the inclusion (\ref{eq:micro_AJ_conjecture}) is weaker than Conjecture \ref{sec:conj_microsupport}.

\section{Gluing properties}

\subsection{Pairing Formula} 

Recall that the half-form line $(\delta, \varphi) $ consists in a complex line $\delta$ together with an isomorphism $\varphi : \delta^2 \rightarrow K_j$. 
For any two Lagrangian subspaces $\nu_1$, $\nu_2$ of $E$ such that $E = \nu_1 \oplus \nu_2$, there exists a unique sesquilinear pairing 
$$ \delta \times \delta \rightarrow \C , \qquad (x_1, x_2) \rightarrow \langle x_1 , x_2 \rangle_{\nu_1, \nu_2} 
$$
satisfying the following two properties: 
\begin{itemize}
\item[-] $ \bigl( \langle u , v \rangle_{\nu_1, \nu_2} \bigr)^2 =   i \frac{ \pi_1 ^* \varphi( u^{ 2}) \wedge  \overline{ \pi_2 ^* \varphi(v^{2})} }{\om}$, for all $u,v \in \delta $,
where  $\pi_1$ and $\pi_2$ are the projections of $E$ onto  $\nu_1$ with kernel $\nu_2$ and onto  $\nu_2$ with kernel $\nu_1$ respectively.
\item[-]  $\langle \cdot , \cdot \rangle_{\nu_1, \nu_2}$ depends continuously on $\nu_1$, $\nu_2$ and for any lagrangian $\nu$, 
$$\langle u, u \rangle_{\nu, j \nu} \geqslant 0, \qquad \forall u \in \delta $$ 
\end{itemize}
The reader is referred to \cite{l1}, Section 6, for more details. 

Consider now two one-dimensional submanifolds $\Gamma^ 1$ and $\Gamma^ 2$ of $E$ intersecting transversally at $x$. 
For $i=1,2$, let $F_i$ be a holomorphic section of $L$ such that its restriction to $L$ is flat with a unitary pointwise norm and $f_i$ be a section of $\delta \rightarrow E$. 

\begin{prop}[\cite{l1}] \label{sec:pairing-formula}
There exists a neighborhood $U$ of $x$ and a sequence $(a_\ell)$ of complex number such that for any $N$, 
\begin{xalignat*}{2}
\Bigl( \frac{k}{2 \pi} \Bigr)^{1/2} \int_U F^k_1 (y)\overline{F}^k_2 (y) & (f_1 (y), f_2 (y))_\delta \; |\om | (y) = \\ &  \Big(\frac{2\pi}{k}\Big)^{1/2}  F^k_1 (x)\overline{F}^k_2 (x)\sum_{\ell = 0 } ^N a_\ell k^{-\ell} + O(k^{-N-3/2} ).
\end{xalignat*}

Furthermore 
$ a_0 = \langle  f_1 (x), f_2 ( x) \rangle_{T_x \Ga^1, T_x \Ga^2 } . $
\end{prop}
The proof is an application of stationary phase lemma. In the sequel we need also to consider the case where $f_1(x) = f_2( x) =0$. Then $a_0$ vanishes and starting from the proof of the above proposition in \cite{l1}, it is easy to compute $a_1$. Write $f_i = g_i s_i$ with $g_i$ a function vanishing at $x$ and $s_i$ a section of $\delta \rightarrow E$. Let $H$ be the bilinear form of $E$   
$$ H(x,y) = \om ( \bar{q}_2 x - q_1 x , y )
$$
where $q_1$ and $q_2$ are the projections onto $E^{0,1} = \{ x + i j x/ x \in E \}$ with kernel $T_x \Ga^1$ and $T_x\Ga^2$ respectively.  $H$ is symmetric non degenerate. Let $G: E \times E \rightarrow \C$ be the Hessian of the product $g_1 g_2$ at $x$. Then 
\begin{gather} \label{eq:deuxieme_coef}
 a_1 =   \frac{i}{2} \Bigl( \sum\limits_{i,j = 1, 2} H^{ij}G_{ij} \Bigr)  \langle  s_1 (x), s_2 ( x) \rangle_{T_x \Ga^1, T_x \Ga^2 }
\end{gather}
where $H^{ij}$ is the inverse of the matrix of $H$ in some basis of $E$ and $G_{ij}$ is the matrix of $G$ in the same basis.  

\subsection{Lens spaces} 
Let us deduce the asymptotic of the WRT invariants of the lens spaces.  Let $p$ and $q$ be two mutually prime integers such that $1< q< p $. Let $\varphi$ be any preserving orientation homeomorphism of $\T ^2 = S^ 1 \times S^ 1$ sending $\{ 1 \} \times S^1$ onto $\{ (z^p, z^q ) / \; z \in S^1\}$ .     The lens space $L(p,q)$ is obtained by gluing two copies $N_1$, $N_2$ of the solid torus $D^2 \times S^1$ along their boundaries $\T^2$:
 $$L(p,q) = N_1 \cup_{ x_1 \sim \varphi (x_2)}  (-N_2) .$$

\begin{theo} \label{sec:Jeffrey_formula}
For any mutually prime integers $p,q$ such that $1< q <p$, there exists sequences $(a_{\ell,n})_n; \ell = 0,1 , \ldots , p-1$ such that for any $N$, 
$$ Z_k ( L(p,q) ) =  e^{i m \pi/4} \sum_{\ell = 0 }^{p-1} k^{m( \ell)} e^{ 2 i \pi \frac{q \ell^2 }{p} k} \sum_{n=0}^{N} a_{\ell, n} k^{-n} + O(k^{-N-1}) $$
where $m$ is an integer and if $\ell =0$ or $p/2$, 
$$ m ( \ell ) = -3/2 , \qquad a_{0,\ell} =  -i \frac{ \sqrt 2 \pi}{p^{3/2}}$$
and otherwise
$$  m ( \ell ) = -1/2, \qquad a_{0, \ell} = \sqrt{\frac{2}{p}}\sin \Bigr( \frac{2 \pi q \ell }{p} \Bigr) \sin \Bigl( \frac{ 2 \pi \ell }{ p} \Bigr).$$
\end{theo}

In particular, we recover Jeffrey's formula {\cite{je}}
$$ Z_k ( L(p,q) ) \sim i \sqrt{\frac{2}{p}} k^{-\frac{1}{2} } \sum_{\ell =1} ^ {p-1} e^{2 i \pi \frac{q \ell^2 }{p} k} \sin \Bigr( \frac{2 \pi q \ell }{p} \Bigr) \sin \Bigl( \frac{ 2 \pi \ell }{ p} \Bigr).$$ 
the $i$ factor corresponding to a canonical pairing. 

\begin{proof} 
It is an application of Proposition \ref{sec:pairing-formula}. We have
$$ Z_k ( L(p,q)) = \langle Z_k ( N_1) , Z_k ( N_2) \rangle _{V_{k} ( \T^2)} 
$$ 
where the boundaries of $N_1$ and $N_2$ are identified with $\T ^2$ through the maps $\operatorname{id}_{\T^2}$ and $\varphi$.  Let us identify as in Section \ref{sec:equiv} the vector space $V_k ( \T^2)$ with the quantization of the torus $E/R$ where $E=H_1 ( \T^2, \R)$ and $ R= H_1 ( \T ^2 , \Z)$.  Let $\mu$ and $\la \in R$ be the homology classes of $S^1 \times \{1\}$ and $\{1 \} \times S^1$. 

By Theorem \ref{theo:solid-torus-state}, $Z_k(N_i)$ is a Lagrangian state supported by $\ga_i \R $ with $\ga_1 = \la $ and $\ga_2 = p \mu + q \la$. In particular the microsupport of $Z_k(N_i)$ is $\ga_i \R + R$. So for any compact neighborhood $C$ of $(\ga_1 \R /\Z) \cap (\ga_2  \R /\Z)$ in $E/R$, we have 
$$ \langle Z_k ( N_1) , Z_k ( N_2) \rangle _{V_{k} ( \T^2)}  = \int_{\tilde C}  (Z_k ( N_1) , Z_k ( N_2)) (x) \mu_M(x)  + O(k^{-\infty})$$
where $\tilde C$ is any lift of $C$ in $E$.
The two circles intersects in $p$ points
$$ (\ga_1 \R/ \Z) \cap (\ga_2 \R/ \Z) = \{ [\ell ( \mu +  q/p \la ) ] ;\; \ell = 0, \ldots , p-1\}.$$ 
By Theorem \ref{theo:solid-torus-state}, we have at $x_\ell = \ell ( \mu +  q/p \la )$, 
$$ Z_k (N_1) (x_\ell  ) =  e^{\frac{i \pi}{4} m_1} \theta_k  \Bigl( \frac{k}{2 \pi} \Bigr)^{1/4} e^{2i \pi k \ell^2 q/p} \sqrt 2 \sin \Bigl( 2 \pi  \frac{\ell q}{p} \Bigr) \Om_{\ga_1}  + O(k^{-\infty} ) $$
and 
$$ Z_k ( N_2) ( x_\ell ) = e^{\frac{i \pi}{4} m_2} \theta_k  \Bigl( \frac{k}{2 \pi} \Bigr)^{1/4} \sqrt 2 \sin \Bigl( 2 \pi  \frac{\ell }{p} \Bigr) \Om_{\ga_2}  + O(k^{-\infty} )  .$$ Let $\pi_1$ and $\pi_2$ be the projection onto $\ga_1 \R$ (resp. $\ga_2 \R$) with kernel $\ga_2 \R$ (resp. $\ga_1\R$). We have
$$ i \frac{(\pi_1^* \Om_{\ga_1}^2 \wedge \pi_2^* \overline{\Om}_{\ga_2}^2 )( \ga_1, \ga_2)}{\om(\ga_1,\ga_2)} = \frac{i}{ \om ( \ga_1, \ga_2)} = \frac{- i}{4 \pi p} .$$
Now the result follows from the pairing formula, cf. Proposition \ref{sec:pairing-formula}. For $\ell = 0 $ and $\ell = p/2$ (if $p$ is even) the two sinus in the formula vanishes. So the corresponding leading order terms vanish and we can compute the next coefficient in the asymptotic expansion with Formula (\ref{eq:deuxieme_coef}). Since the result does not depend on the choice of the complex structure, one may assume that $j \mu = \la$. 
Write $f_i = g_i \Om_{\ga_i}$. Then the Hessian $G$ of $g_1 g_2$ and the bilinear form $H$ are given respectively in the basis $\mu, \la$ by 
$$ 
G = \frac{8 \pi^2}{q+ ip} \begin{pmatrix} 1&0\\ 0&1\end{pmatrix}, \qquad H = \frac{4 i \pi}{  i + q/p}  \begin{pmatrix} i + 2 q/p  & -1 \\ -1 &  i  \end{pmatrix}.
$$
The final result follows easily. 
\end{proof}

\subsection{Melvin-Morton-Rozansky conjecture} 

Let $K$ be any knot in $S^3$. The representations $\rho \in \mo(E_K)$ such that $\tr(\rho(\mu)) $ is close to 2 are abelian. More precisely, we have the following lemma.

\begin{lem}\label{lem:borne}
For any knot $K$ there exists $\epsilon>0$ such that 
\begin{gather} \label{eq:bande_sans_irr}
\pi^{-1}\bigl( r(\mo^{\ir}(E_K)) \bigr) \cap \bigl(  (-\ep,\ep)\la+  \Z \mu \bigr) = \emptyset
\end{gather}
\end{lem}

\begin{proof} First observe that the only element in $\su$ with trace equal to 2 is the identity. Furthermore  $\pi_1(E_K)$ is normally generated by the meridian. So the only representation $\rho \in  \mo(E_K)$ satisfying $\tr(\rho(\mu))=2$ is the trivial representation. Second, since 1 is not a root of the Alexander polynomial of $K$, the trivial representation is not a limit of irreducible representations. We conclude  easily using that $\mo (E_K)$ is compact.
\end{proof}

Necessarily, $\ep < 1/2 $. Indeed, the moduli space $\mo^{\ir}(E_K)$ is not empty (\cite{km}) and $\pi^{-1}\bigl( r(\mo^{\ir}(E_K)) \bigr)$ is preserved by the translation of vector $\la /2$ (section 3.1.3 of \cite{LJ2}).  

\begin{theo} \label{theo:MMR}
Let $K$ be a knot satisfying Conjectures \ref{sec:conj_microsupport} and \ref{conj:abelian}. Let $\epsilon >0 $ satisfying (\ref{eq:bande_sans_irr}) and such that $\Delta_K( e^{2i \pi q }) \neq 0$ for all $q \in ( - \ep, \ep)$. 
Then for any $\delta \in ( 0, \epsilon )$, there exists $C$ such that for any integers $k$ and $\ell$ satisfying $|\ell |  \leqslant k \delta $, we have  
$$ \Bigl| \langle Z_K, \Psi_\ell \rangle -   \frac{e^{i m\frac\pi 4}}{2k^{1/2}}  \frac{ \si - \si^{-1} }{ \Delta_K ( \si^2)} \Bigr| \leqslant C k ^{-3/2}  
$$ 
where $\si = \exp ( i \pi \ell /k)$, $\Delta_K$ is the Alexander polynomial of $K$ and $m$ is a fixed integer.
\end{theo}

\begin{proof} 
Let $\delta$ and $\delta'$ be such that $0 < \delta < \delta' <\epsilon$. Using the estimates of Proposition \ref{sec:base_lagrangian_states} together with the fact that $\Psi_\ell = T^*_{\ell \la/2k} \Psi_0$ we prove that there exists a sequence $(C_N)$ of positive numbers satisfying the following: for any $\ell$ and $k$ such that $|\ell |  \leqslant k \delta $ and for any $N$, we have for all $x \notin (-\delta', \delta') \la + \R \mu$ modulo $R$
\begin{gather} \label{eq:neg}
 | \Psi_\ell (x) | \leqslant C_N k^{-N} 
\end{gather}
and for all $x \in (-\delta', \delta') \la + \R \mu$
\begin{gather} \label{eq:lag}
 \bigl| \Psi_\ell (x)  - \Bigl( \frac{k}{2\pi} \Bigr)^{1/4} T^*_{\ell \la / 2k } t_\mu^k (x) \otimes \Om_{\mu}  \bigr|  \leqslant C_N k^{-N}. 
\end{gather}
The assumption on $\ep$, Conjectures \ref{sec:conj_microsupport} and \ref{conj:abelian}  imply that $Z_K = Z^{\operatorname{ab}}_K + O(k^{-\infty})$  on $ (-\epsilon, \epsilon) \la + \R \mu$ with 
 $$Z^{\operatorname{ab}}_K = \Bigl( \frac{k}{2 \pi} \Bigr)^{1/4}  i^{m} t_\la^k(x) \otimes \Om_\la f(x ,k ) $$
where $f(\cdot, k)$, $t_\la$ and $\Om_\la$ are as in the Conjecture  \ref{conj:abelian}. Using this and (\ref{eq:neg}), we obtain that for any $\ell$ and $k$ satisfying $|\ell |  \leqslant k \delta $ 
$$ \langle Z_K , \Psi_\ell \rangle =  \int_D  (Z_K^{\operatorname{ab}}(x) , \Psi_\ell(x) )_{\delta} |\om|(x)   + O(k^{-\infty} ) $$
with the $O$ uniform with respect to $k$ and $\ell$ and the domain $$D = [-\delta', \delta']\la + [-\tfrac{1}{2} , \tfrac{1}{2} ] \mu.$$ 
To estimate this integral, we use  Equation (\ref{eq:lag}) and apply Proposition \ref{sec:pairing-formula}. Actually we need a slightly improved version where one of the Lagrangian submanifolds depends on a parameter. Here $\Ga^1 = \R \la$, $\Ga^2 =   - \frac{\ell}{2k}  \la  + \R \mu$ and Proposition \ref{sec:pairing-formula} extends easily with a result uniform for $\frac{\ell}{2k}$ running over any compact domain.  
So we obtain 
$$ \langle Z_K , \Psi_\ell \rangle = \Bigl( \frac{2\pi}{k} \Bigr)^{1/2} \bigl( t_\la (y)\overline{t}_\mu(y) \bigr)^k    a  f ( y,k) + O(k^{-3/2})  
$$
where $y = - \frac{\ell}{2k}  \la$ is the intersection point of $\R \la$ and  $-\frac{\ell}{2k}  \la  + \R \mu$. The $O$ is uniform with respect to $k$ and $\ell$ satisfying $ |\ell | \leqslant k \delta$. The coefficient $a$ is a square root of 
$$   i  \frac{ \Om_\la^{2} ( \la )   \overline{\Om_\mu}^{2}( \mu) }{\om (\la, \mu)} = - \frac{i}{4 \pi}  $$ 
Furthermore 
$$ f( y, k ) =  \frac{1}{\sqrt 2}   \frac{ \si - \si^{-1}  }{\Delta_K (\si^2)}  + O(k^{-1}) 
$$ 
with $\si  = \exp ( i \pi \ell /k ) $. Finally $t_\la (y) = t_\mu(y)=1$ which concludes the proof. 
\end{proof}

Let us compare this result with the Melvin-Morton-Rozansky theorem proved in \cite{bng} and \cite{ro}. Recall that from Equations \eqref{eq:normalization} and \eqref{eq:etat_noeud} one gets 
$$\langle Z_K,\psi_\ell\rangle=\frac{\sin{\pi/k}}{\sqrt{k}}J_\ell^K(-e^{i\pi/2k})=\tilde{J}_\ell^K(e^{-2i\pi/k})\frac{\sin(\pi\ell/k)}{\sqrt{k}}$$
Hence from Theorem \ref{theo:MMR} we have, 
\begin{xalignat*}{2} 
 \tilde{J}_\ell^K(e^{-2i\pi/k}) = & \frac{ie^{im\pi/4}}{\Delta(e^{2i\pi\ell/k})} + | \sin ( \pi \ell /k )|^{-1} O (k^{-1}) \\
= & \frac{ie^{im\pi/4}}{\Delta(e^{2i\pi\ell/k})} + O( \ell^{-1} ) 
\end{xalignat*}
where the $O$'s are uniform with respect to $k$ and $\ell$ satisfying $1 \leqslant  \ell \leqslant  k \delta $. Here we used that $|\sin ( \pi x /2) | \geqslant x $ for all $x \in [0,1] $. 
 
Comparing with Theorem 1.3  of \cite{glmmr}, we have that $ie^{im\pi/4}=1$. Furthermore the regime $1 \leqslant  \ell \leqslant  k \delta $ is the same as in \cite{glmmr}. But our proof gives an upper bound on $\delta$, that is $\delta < \ep $ where $\ep$ satisfies Lemma 
\ref{lem:borne}. We will see in Theorem 4.6 in \cite{LJ2} what happens when $\ell/k$ overpasses this bound.

\section{Appendix} 

\subsection{Schwartz kernel of Toeplitz operators}
Consider as in section a 2-dimensional symplectic vector space $(E, \om)$ with a lattice $R$ of volume $2 \pi$, a linear complex structure $j$ and a half-form line $\delta$. Recall that we lifted the action of $R$ on $E$ to the bundle $L$ using the Heisenberg group. Let $M$ be the torus  $E/ R$, $L_M \rightarrow M$ be the quotient of $L$ by $R$ and $\delta_M \rightarrow M$ be the trivial line bundle over $M$ with fiber $\delta$. Then the space $\Hilb_k = \Hilb_k(j, \delta)$ is identified with the space $H^0 ( M, L^k_M \otimes \delta_M)$ of holomorphic sections of $L^k_M  \otimes \delta_M$. 

The scalar product of $\Hilb_k$  gives an
isomorphism $ \operatorname{ End} (\Hilb_k) \simeq \Hilb_k \otimes  \overline {\Hilb_k}$. The latter space can be regarded as the space of holomorphic sections of
$$( {L}^k_M \otimes \delta_M ) \boxtimes  ( \overline{L}^k_M \otimes  \overline{\delta}_M ) \rightarrow  M \times \overline{M} .$$ The section associated in this way to an endomorphism is its Schwartz kernel.  

Let $F$ be a holomorphic section of $ L_M \boxtimes \overline{L}_M
 \rightarrow  M \times \overline M$ defined on a neighborhood $U$ of the diagonal such that for all $x \in M$, $F(x, x) = u \otimes \overline{u} $ where $u$ is any normalized vector in $L_{M,x}$. Restricting $U$ if necessary, we have that the pointwise norm of $F$ is $\leqslant 1 $ with equality exactly on the diagonal.

Recall that we defined Toeplitz operators of $\Hilb_k$ and their symbols in Section \ref{sec:tore_Toeplitz_operators}. The following characterization of Toeplitz operators in terms of their Schwartz kernel has been proved in \cite{l5}. It extends actually to any compact K{\"a}hler prequantizable manifold.

\begin{theo} \label{theo:schw-kern-toepl}
A family $(S_k \in \operatorname{ End} (\Hilb_k), \; k \in \Z_{>0} )$ is a Toeplitz operator if and only if its Schwartz kernel $(S_k ( \cdot))$ satisfies the following conditions:
\begin{itemize} 
\item $(S_k (\cdot))$  is a $O(k^{-\infty})$ uniformly on any compact set of $T^2$  which does not meet the diagonal.
\item We have over $U$  
$$ S_k ( x, y) = \Bigl( \frac{k}{2 \pi } \Bigr) F^k ( x, y)  f ( x, y , k ) + O(k^{-\infty}). $$ where $f(\cdot , k)$ is a sequence of $\Ci(U, \delta \otimes \bar{\delta} )$ which admits an asymptotic expansion of the form $f_0 + k^{-1} f_1 + k^{-2} f_2 + \ldots$  for the $\Ci$ topology.
\end{itemize}  
Furthermore, the principal and subprincipal symbols of $(S_k)$ are respectively $g_0$ and $g_1 - \frac{1}{2} \Delta g_0$, where $g_0$, $g_1$ are the restrictions of $f_0$ and $f_1$ to the diagonal and $\Delta$ is the holomorphic Laplacian of $M$.   
\end{theo}

To be more precise, let $\De : M \rightarrow M^2$ be the  diagonal embedding. Then $\De ^ * ( \delta \boxtimes \bar  \delta) = \delta \otimes \bar \delta $ which is canonically isomorphic to the trivial line bundle over $M$ because $\delta$ is hermitian. So  the restriction to the diagonal of sections of $\delta \boxtimes \bar \delta  \rightarrow M^2$ can be identified with functions. Furthermore the holomorphic Laplacian is 
\begin{gather} \label{eq:def_laplacien}
 \Delta = \frac{ \tau - \bar { \tau} } { 2 i \pi } \partial_{\zeta} \partial_{\bar{ \zeta}} 
\end{gather}
where $\tau$ is the parameter of the complex structure defined as in Theorem \ref{sec:rep_Heisenberg} and $\zeta$ is the holomorphic coordinate $p + \tau q$ of $E$.

\subsection{Proof of theorem \ref{sec:Toeplitz_operators}}  

We will deduce Theorem \ref{sec:Toeplitz_operators} from Theorem \ref{theo:schw-kern-toepl}. We lift every function or section defined on a neighborhood of the diagonal of $M^2$ to a neighborhood of the diagonal of $ E^2$. In particular the section $F$ lifts to 
$$ F  =  \exp \Bigl(  - \frac{2 i \pi }{\tau - \bar{\tau} }  ( \zeta_1 - \bar{\zeta}_2 )^2 \Bigr)  t \boxtimes \bar{t} $$
Here $t$ is the holomorphic section of $L \rightarrow E$ defined in  (\ref{eq:def_t}), the complex $\tau$ parameterizes the complex structure as in Theorem \ref{sec:rep_Heisenberg}, $\zeta_1$ and $\zeta_2$ are the pull-back of the holomorphic coordinate $\zeta = p + \tau q$ of $E$ to the first and second factor of $E^2$.     
It is known that the Schwartz kernel of the projector $\Pi$ lifts to $\bigl( \frac{k}{2 \pi} \bigr) F  + O(k^{-\infty}) $ on a neighborhood of the diagonal.  This can be proved by a direct computation using the orthogonal basis introduced in Theorem \ref{sec:rep_Heisenberg} and Poisson summation formula (cf. proof of Theorem in 8.2 \cite{l2}). This follows also at first order from Theorem  \ref{theo:schw-kern-toepl}. 

The Schwartz kernel of $T^*_{\nu/k}$ is $\bigl( \frac{k}{2 \pi} \bigr) T^*_{\nu /k} F^k $ where $T^*_{\nu /k}$ acts on the first factor. 
It follows from Equations (\ref{eq:pull_back}) and (\ref{eq:def_t}) that
\begin{xalignat*}{2} 
 T_{\nu / k }^* t ^k = & \exp \bigl(  - 2 i \pi ( \dot p q - \dot q p )  +  2 i \pi k ( \zeta + \dot \zeta/k )( q + \dot q /k ) \bigr)  \\ 
= & \exp \bigl(  2 i \pi ( 2 \zeta \dot q + \dot \zeta \dot q / k ) \bigr) t^k 
\end{xalignat*}
where $\nu = \dot p \mu + \dot q \la $ and $\dot \zeta = \dot p + \tau \dot q $. So the 
  Schwartz kernel of $T^*_{\nu/k}$  lifts on a neighborhood of the diagonal to 
$$ \Bigl(  \frac{k}{2 \pi} \Bigr) F^k \si ( \zeta_1, \bar{\zeta}_2 ) + O(k^{-\infty}) $$
 with 
\begin{xalignat*}{2} 
\si ( \zeta, \bar{ \zeta})  = &  \exp \Bigl(  -\frac{2 i \pi }{\tau - \bar{\tau}} ( 2 \dot \zeta ( \zeta - \bar \zeta ) + \dot \zeta^2/k  ) +  2 i \pi ( 2 \zeta \dot q + \dot \zeta \dot q / k ) \Bigr) \\
= &  \exp \Bigl( - 4 i \pi \dot \zeta q - \frac{ 2 i \pi }{ \tau - \bar { \tau}} \dot \zeta ^2 / k + 4 i \pi \zeta \dot q + 2 i \pi \dot \zeta \dot q / k \Bigr) 
\\
= & \exp \Bigl( 4 i \pi ( p \dot q - \dot p q ) - \frac{2 i \pi }{ \tau - \bar{ \tau}}  \bigl| \dot \zeta \bigr|^2  / k  \Bigr)  
\end{xalignat*}
Finally it follows from (\ref{eq:def_laplacien}) that
$ \Delta \si = - \frac{ 2 i \pi }{ \tau - \bar \tau } \bigl| \dot \zeta \bigr|^2  \si $ so that
$$ \si - \frac{1}{2 k } \Delta \si =  \exp \bigl( 4 i \pi ( p \dot q - \dot p q ) \bigr)  + O(k^{-2} ) $$ 
which concludes the proof.


\begin{thebibliography}{10}

\bibitem[A05]{an}
J. E. Andersen.
\newblock Deformation quantization and geometric quantization of abelian moduli spaces.
\newblock {\em Comm. Math. Phys}, \textbf{255}, no. 3, (2005), 727-745. 

\bibitem[BNG96]{bng}
D.~Bar-Natan and S.~Garoufalidis.
\newblock On the Melvin-Morton-Rozansky conjecture.
\newblock {\em Invent. Math.}, \textbf{125}, (1996), 103-133.

\bibitem[BHMV95]{bhmv}
C. Blanchet, N. Habegger,  G. Masbaum and P. Vogel.
\newblock Topological quantum field theories derived from the Kauffman bracket.
\newblock {\em Topology}, \textbf{34}, (1995), 883-927.

\bibitem[C03]{l5}
L.~Charles.
\newblock Berezin-Toeplitz operators, a semi-classical approach, 
\newblock {\em Comm. Math. Phys.}, \textbf{239}, no. 1-2, (2003), 1-28.


\bibitem[C03]{l3}
L.~Charles.
\newblock Quasimodes and Bohr-Sommerfeld conditions for the Toeplitz operators, 
\newblock {\em Comm. Partial Differential Equations}, \textbf{28}, Vol. 9-10, (2003), 1527-1566.

\bibitem[C06]{l4}
L.~Charles.
\newblock Symbolic calculus for Toeplitz operators with half-forms.
\newblock {\em Journal of Symplectic Geometry}, \textbf{4}, Vol 2, (2006), 171-198.

\bibitem[C10a]{l1}
L.~Charles.
\newblock On the Quantization of Polygon Spaces. 
\newblock {\em Asian Journ. of Math.}, \textbf{14}, Vol. 1, (2010), 109-152.

\bibitem[C10b]{l2}
L.~Charles.
\newblock Asymptotic properties of the quantum representations of the modular group. 
\newblock to appear in {\em Trans. Amer. Math. Soc.}, 2010.

\bibitem[C11]{Ctoric}
L.~Charles
\newblock Torus knot state asymptotics.
\newblock in preparation.

\bibitem[CM11]{LJ2}
L.~Charles and J.~March{\'e}.
\newblock Knot state asymptotics II. Witten conjecture and irreducible representations.



\bibitem[CCGLS94]{ccgls}
D. Cooper, D, M. Culler, H. Gillet, D. Long and P. Shalen.
\newblock Plane curves associated to character varieties of 3-manifolds.
\newblock {\em Invent. Math.}, \textbf{118}, (1994), 47-84.

\bibitem[D74]{duistermaat} 
J. J. Duistermaat, 
\newblock Oscillatory integrals, Lagrange immersions and unfolding of singularities.
\newblock {\em Comm. Pure Appl. Math.}, \textbf{27}, (1974), 207-281. 


\bibitem[G98]{gar1}
S.~Garoufalidis.
\newblock Applications of TQFT invariants in low dimensional topology
\newblock  {\em Topology}, \textbf{37}, (1998), 219-224.

\bibitem[G04]{gar}
S.~Garoufalidis.
\newblock On the characteristic and deformation varieties of a knot.
\newblock Proceedings of the Casson Fest, {\em Geometry and Topology Monographs}, \textbf{7}, (2004), 91-309.

\bibitem[GL05]{gl}
S.~Garoufalidis and T. T. Q. Le.
\newblock The colored Jones function is q-holonomic.
\newblock {\em Geom. Topol}, \textbf{9}, (2005), 1253-1293.

\bibitem[GL]{glmmr}
S.~Garoufalidis and T. T. Q. Le.
\newblock Asymptotics of the colored Jones function of a knot.
\newblock Arxiv: math/0508100.

\bibitem[GS10]{gs} 
S. Garoufalidis and X. Sun.
\newblock The non-commutative A-polynomial of twist knots.
\newblock {\em Journal of Knot Theory and its Ramifiations}, \textbf{19} (2010), 1571-1595.


\bibitem[GU10]{gu} 
R. Gelca and A. Uribe.
\newblock From classical theta functions to topological quantum field theory.
\newblock arxiv:1006.3252




\bibitem[GM10]{gm}
P.~Gilmer and G.~Masbaum.
\newblock  Maslov index, Lagrangians, Mapping Class Groups and TQFT. 
\newblock to appear in {\em Forum Mathematicum}, arXiv:0912.4706, 2010.

\bibitem[Ha01]{habiro}
K.~Habiro.
\newblock On the quantum sl(2) invariants of knots and integral homology spheres.
\newblock In: Invariant of knots and 3-manifolds (Kyoto 2001), {\em Geometry and Topology Monographs}, Vol. 4, (2002), 55-68.

\bibitem[Hi04]{hikami}
K.~Hikami.
\newblock Difference equation of the colored Jones polynomial for torus knot.
\newblock {\em Internat. J. Math.}, \textbf{15}, no.9, (2004), 959-965.

\bibitem[H{\"o}90]{hormander} 
L. H{\"o}rmander.
\newblock The analysis of linear partial differential operators. I.
\newblock Grundlehren der Mathematischen Wissenschaften, Springer-Verlag, {\bf 256}, (1990). 

\bibitem[J92]{je}
L.~C.~Jeffrey.
\newblock Chern-{S}imons-{W}itten invariants of lens spaces and torus bundles,
  and the semiclassical approximation.
\newblock {\em Comm. Math. Phys.}, \textbf{147}, no.3, (1992), 563-604.

\bibitem[KM04]{km}
P. B. Kronheimer and T.S. Mrowka.
\newblock Witten's conjecture and property P. 
\newblock {\em Geom. Topol.}, \textbf{8}, (2004), 295-310. 

\bibitem[Le06]{le}
T.~Q.~T.~Le.
\newblock The Colored Jones Polynomial and the A-Polynomial of Knots.
\newblock {\em Adv. in Math.}, \textbf{207}, (2006), 782-804.

\bibitem[Ma03]{masbaum}
G.~Masbaum.
\newblock Skein-theoretical derivation of some formulas of Habiro.
\newblock {\em Algebraic and Geometric Topology}, \textbf{3}, (2003), 537-556.

\bibitem[Mo95]{morton}
H.~C.~Morton.
\newblock The coloured Jones function and Alexander polynomial for torus knots.
\newblock {\em Math. Proc. Cambridge Philos. Soc.}, \textbf{117}, no.1, (1995), 129-135.

\bibitem[Mu83]{mumford}
D. Mumford. 
\newblock Tata lectures on theta. I.
\newblock Progress in Mathematics, {\bf 28}, Birkh{\"a}user Boston, Inc., Boston, MA, (1983).

\bibitem[RT91]{rt}
N. Reshetikhin and V. G. Turaev.
\newblock Invariants of {$3$}-manifolds via link polynomials and quantum groups.
\newblock {\em Invent. Math.}, \textbf{103}, no. 3, (1991), 547-597.

\bibitem[R98]{ro}
L.~Rozansky.
\newblock The universal R-matrix, Burau representation and the Melvin-Morton expansion of the colored Jones polynomial
\newblock {\em Adv. Math.}, \textbf{134}, no. 1, (1998), 1-31.

\bibitem[S96]{sorger} 
C. Sorger.
\newblock La formule de Verlinde
\newblock {\em Ast{\'e}risque}, \textbf{237}, no. 3, (1996), 87-114.

\bibitem[W89]{witten}
E.~Witten.
\newblock Quantum field theory and the Jones polynomial.
\newblock {\em Comm. Math. Phys.}, \textbf{121}, no. 3, (1989), 351-399.

\end{thebibliography}
\end{document}